\documentclass[12pt]{article}
\usepackage{amsmath,amssymb,latexsym}
\def\Z{{\mathbb Z}}
\def\mod{{\rm mod \hspace{1mm}}}

\def\Tr{{\rm Tr}}

\def\Or{{\rm Or}}

\def\gcd{{\rm gcd}}
\def\GL{{\rm GL}}
\def\GO{{\rm GO}}
\def\SO{{\rm SO}}
\def\SL{{\rm SL}}

\def\Cl{{\rm Cl}}

\def\P{{\mathbb P}}
\def\Disc{{\rm Disc}}

\def\Vol{{\rm Vol}}

\def\C{{\mathbb C}}
\def\F{{\mathbb F}}

\def\Q{{\mathbb Q}}

\def\Z{{\mathbb Z}}
\def\P{{\mathbb P}}
\def\F{{\mathbb F}}
\def\Q{{\mathbb Q}}
\def\O{{\mathcal O}}

\newtheorem{theorem}{Theorem}
\newtheorem{lemma}[theorem]{Lemma}

\newtheorem{proposition}[theorem]{Proposition}
\newtheorem{remark}[theorem]{Remark}
\newtheorem{example}[theorem]{Example}

\newenvironment{proof}{\noindent {\bf Proof: }}{$\Box$ \vspace{2 ex}}
\newenvironment{proof1}{\noindent {\bf Proof of Theorem~\ref{cubflds}: }}{$\Box$ \vspace{2 ex}}
\newenvironment{proof2}{\noindent {\bf Proof of Theorem~\ref{resolvents}: }}{$\Box$ \vspace{2 ex}}
\newenvironment{proof3}{\noindent {\bf Proof of Theorem~\ref{pure}: }}{$\Box$ \vspace{2 ex}}

\title{On the number of cubic orders of bounded discriminant having 
automorphism group $C_3$, and related problems}

\author{Manjul Bhargava and Ariel Shnidman}

\begin{document}
\setcounter{tocdepth}{2}
\maketitle

\begin{abstract}
  For a binary quadratic form $Q$, we consider the action of $\SO_Q$
  on a two-dimensional vector space.  This representation yields
  perhaps the simplest nontrivial example of a prehomogeneous vector space that is not irreducible, and of a {coregular space} 
  whose underlying
  group is not semisimple.  We show that the nondegenerate integer
  orbits of this representation are in natural bijection with orders in cubic fields
  having a fixed ``lattice shape''.  Moreover, this correspondence is
  {\it discriminant-preserving}: the value of the invariant polynomial of an
  element in this representation agrees with the discriminant of the
  corresponding cubic order.

  We use this interpretation of the integral orbits to solve three
  classical-style counting problems related to cubic orders and
  fields.  First, we give an asymptotic formula for the number of
  cubic orders having bounded discriminant
  and nontrivial automorphism group.  More generally, we give an
  asymptotic formula for the number of cubic orders that have bounded
  discriminant and any given lattice shape (i.e., reduced trace form, up to scaling).  Via a
  sieve, we also count cubic {\it fields} of bounded
  discriminant whose rings of integers have a given lattice shape.
   We find, in particular, that among cubic orders (resp.\ fields) having lattice shape 
   of given discriminant $D$, the shape is {\it equidistributed} 
   in the class group $\Cl_D$ of binary quadratic forms of discriminant~$D$.
As a by-product, we also obtain an asymptotic formula for the number of cubic
fields of bounded discriminant having any given quadratic resolvent~field.
\end{abstract}

\tableofcontents

\section{Introduction}

%

An order in a cubic field (or {\it cubic order} for short) has
either 1 or 3 automorphisms.  The number of cubic orders (resp.\
fields) with trivial automorphism group and bounded discriminant was
computed asymptotically in the classical work of Davenport--Heilbronn
\cite{DH}.\footnote{When referring to the number of cubic orders or fields 
with a given property, we always mean the number of such objects {up to isomorphism}.}  
A corresponding asymptotic formula for the number of cubic
fields with an automorphism of 
order 3 (called {\it $C_3$-cubic fields}) was obtained by Cohn \cite{cohnab}, but a formula for $C_3$-cubic \textit{orders} has not previously been
obtained.  In this article, we prove the following theorem:

\begin{theorem}\label{c3ordercount}
  The number of cubic orders having automorphism group isomorphic to
  a cyclic group of order $3$, and discriminant less than $X$,
  is 
\[ \frac{\pi}{6\sqrt{3}} X^{1/2} + O(X^{1/4}) .\]
\end{theorem}

More generally, we prove asymptotics for the number of cubic orders
having any given ``lattice shape''.  To be more precise, a
\textit{cubic ring} is a commutative ring with unit that is free of
rank 3 as a $\mathbb{Z}$-module.  Such a ring $R$ is endowed with a
linear map $\Tr: R \rightarrow \mathbb{Z}$ called the \textit{trace},
which sends $z \in R$ to the trace of the endormorphism $\times z:R \rightarrow
R$ defined by multiplication by $z$.  The \textit{discriminant}
$\Disc(R)$ of a cubic ring $R$ with $\mathbb{Z}$-basis
$\alpha_1,\alpha_2,\alpha_3$ is defined to be
$\det(\Tr(\alpha_i\alpha_j)) \in \mathbb{Z}$.  A cubic \textit{order} is a cubic ring that is also an integral domain.    

For a cubic ring $R$, the restriction of the trace form $\Tr(z^2)$ to the trace-zero part of
$\Z + 3R$ is an integer-valued binary quadratic form.  
If $R$ has nonzero discriminant
then, via a choice of basis, this form can be written as~$nQ(x,y)$, where
$Q$ is a primitive integral binary quadratic form and $n$ is a positive integer.  We define the
\textit{shape} of $R$ to be the~$\GL_2(\mathbb{Z})$-equivalence class
of the binary quadratic form $Q(x,y)$.  Since it is often convenient, we will  usually refer to $Q$ itself (or an equivalent form) as the shape of $R$.\footnote{We 
note that the shape of $R$ may also be 
described in terms of the restriction of the trace form $\Tr(z^2)$ to the {\it projection} of the lattice $R$ onto the plane in $R\otimes \Q$ that is orthogonal to 1.  This yields the binary quadratic form
$(n/3)Q(x,y)$, and so scaling by $n/3$ again gives the primitive integral binary quadratic form~$Q$ which we call the shape. We prefer to use our definition in terms of $\Z+3R$, as then we can work integrally and do not have to refer to $R\otimes\Q$.}

If $Q$ is a primitive integral binary quadratic form, then we define
$N_3(Q,X)$ to be the number of cubic orders having shape $Q$ and absolute 
discriminant less than $X$. It is easy to see~that, by definition, the shape $Q$ of a cubic ring cannot be negative definite;
hence all quadratic forms $Q$ in this paper  are assumed to be either positive definite or indefinite.  The following theorem gives an asymptotic
formula for $N_3(Q,X)$ as~$X\to\infty$.  
\begin{theorem}\label{cubrings} 
  Let $Q$ be a primitive integral binary quadratic form with non-square
  discriminant $D$.  Set $\alpha=1$ if $3\mid D$
  and $\alpha=0$ otherwise. Set $\beta = 1$ if $D > -4$ and $\beta = 0$
  otherwise.  Set $\gamma = 1$ if $Q$ is 
ambiguous\footnote{Recall that a quadratic form $Q(x,y) =
  rx^2+sxy+ty^2$ is said to be \textit{ambiguous} if
there is an automorphism of $Q$ in $\GL_2(\mathbb{Z})$ with
determinant $-1$.  Equivalently, $Q$ is ambiguous if it is
$\SL_2(\mathbb{Z})$-equivalent to the form $Q' = rx^2-sxy+ty^2$.}
 and $\gamma = 0$
  otherwise. Then $$N_3(Q,X) = \frac{3^{\alpha+\beta-\frac{3}{2}}
    \cdot L(1,\chi_D)}{2^\gamma \cdot
    h(D)\sqrt{|D|}}X^{1/2}+O(X^{1/4}).$$
\end{theorem}
Here, $L(s,\chi_D)$ is the Dirichlet $L$-function associated to the
primitive quadratic character $\chi_D$ of conductor $D$ and $h(D)$
denotes the size of the narrow class group of binary quadratic forms of
discriminant $D$ up to $\SL_2(\Z)$-equivalence.

A cubic ring has three automorphisms if and only if its shape is
equivalent to the quadratic form $Q(x,y) = x^2+xy+y^2$ (see proof
of Theorem \ref{Qorb}). Thus, for this choice of~$Q$, the quantity $N_3(Q,X)$ is the
number of cubic orders with discriminant less than $X$ that have three
automorphisms, and Theorem~\ref{c3ordercount} follows from
Theorem~\ref{cubrings}.

The main term in Theorem \ref{cubrings} is nearly a function 
of the discriminant $D$ of $Q$; only the factor of $2^\gamma$ depends on the particular equivalence class of $Q$.  With this in mind, we introduce the notion of an
\textit{oriented cubic ring}, which is a pair $(R,\delta)$ consisting
of a cubic ring $R$ and an isomorphism $\delta: \wedge^3R \to
\mathbb{Z}$.  
We usually refer to an oriented
cubic ring $(R,\delta)$ simply as $R$, with the accompanying
isomorphism $\delta$ being implied.  The {\it shape} of an oriented cubic
ring $R$ is defined as before, but now using oriented bases and 
$\SL_2(\mathbb{Z})$-equivalence.   

We define $N_3^{\Or}(Q,X)$ to be the number of isomorphism classes of
oriented cubic orders having shape $Q$ and absolute discriminant less
than $X$.  Notice that $Q(x,y)$ is ambiguous if and only if its
$\GL_2(\Z)$-equivalence class coincides with its
$\SL_2(\Z)$-equivalence class.  If $Q$ is not ambiguous then its
$\GL_2(\Z)$-class splits into two $\SL_2(\Z)$-classes.  In other
words, $$ N_3^\Or(Q,X) = 2^\gamma N_3(Q,X),$$ where $\gamma$ is
defined as in
Theorem \ref{cubrings}.  Thus Theorem \ref{cubrings} is equivalent to the following:
\begin{theorem}\label{orcubrings} 
  Let $Q$ be a primitive integral binary quadratic form with non-square
  discriminant $D$.  Set $\alpha=1$ if $3\mid D$
  and $\alpha=0$ otherwise, and set $\beta = 1$ if $D > -4$ and $\beta = 0$
  otherwise.  Then $$N_3^{\Or}(Q,X) =
  \frac{3^{\alpha+\beta-\frac{3}{2}} \cdot
    L(1,\chi_D)}{h(D)\sqrt{|D|}}X^{1/2}+O(X^{1/4}).$$
\end{theorem}
In particular, among oriented cubic orders with shape of discriminant $D$, the shape is  \textit{equidistributed} in the class group $\Cl_D$ of binary quadratic forms of discriminant~$D$.  

We note that the exponent $1/4$ in the error term in Theorem~\ref{orcubrings} is optimal.  If we count cubic \textit{rings} (instead of just cubic orders) having shape $Q$, then we find that the main term in Theorem~\ref{orcubrings} stays the same (i.e., the number of cubic rings of a given shape $Q$ that are not orders in cubic fields is negligible), but the error term becomes smaller.   

Via a suitable sieve, we use Theorem~\ref{orcubrings} to determine
asymptotics for the number $M_3(Q,X)$ (resp.\ $M_3^\Or(Q,X)$) of {\it
  maximal} cubic orders (resp.\ maximal oriented cubic orders) having
shape $Q$ and discriminant bounded by $X$.  Thus $M_3(Q,X)$ is the
number of cubic fields with absolute discriminant less than $X$ whose
rings of integers have shape $Q$.  As before we have $M_3^\Or(Q,X) =
2^\gamma M_3(Q,X) $, where $\gamma = 1$ if $Q$ is ambiguous and
$\gamma = 0$ otherwise.

\begin{theorem}\label{cubflds}
  Let $Q$ be a primitive integral binary quadratic form of discriminant
  $D$.  Suppose that either $D$ or $-D/3$ is a non-square fundamental discriminant.  
  Set $\alpha = 1$ if $3\,|\,D$ and $\alpha = 0$ otherwise.  Also, set $\beta = 1$ if $D > -4$ and $\beta = 0$ otherwise.  
  Then 
$$\!M_3^\Or(Q,X)=\frac{3^{\alpha+\beta+\frac{1}{2}}\mu_3(D)}{4\pi^2\sqrt{|D|}}\cdot
 \frac{ L(1,\chi_D)}{h(D)}\!\!\prod_{\substack{\left(\frac{D}{p}\right)=1,\\p\neq3}}\!\!\!\!\left(1-\frac{2}{p(p+1)}\right)\prod_{\substack{p | D,\\p\neq3}}\left(\frac{p}{p+1}\right)X^{1/2}+o(X^{1/2})$$
  where $$\mu_3(D) = \begin{cases}
16/27 & \text{if $3\;\!\nmid\!\: D$},\\
22/27 & \text{if $3 \parallel D$},\\
2/3 & \text{if $9 \parallel D$}.
\end{cases}$$
For all other non-square values of $D$, we have $M_3(Q,X)=0$.
\end{theorem}    

As in Theorem~\ref{orcubrings}, we see that the shapes of rings of integers in
oriented cubic fields (when ordered by absolute discriminant) are {\it
  equidistributed} in the respective class groups $\Cl_D$ of binary
quadratic forms of discriminant~$D$.  The error term in Theorem~\ref{cubflds} can certainly be improved, although we shall not
investigate the issue in this paper.

Applying Theorem \ref{cubflds} to the form $Q(x,y) = x^2+xy+y^2$ yields
the following result of Cohn \cite{cohnab} on the number of abelian
cubic fields having bounded discriminant (though our methods are
completely different!).

\begin{theorem}[Cohn]\label{cohnthm}
The number of abelian cubic fields having discriminant less than $X$ is
\[ \frac{11\sqrt{3}}{36\pi}\prod_{p\equiv1(3)}\Bigl(1 -
\frac2{p(p+1)}\Bigr)\cdot X^{1/2}+o(X^{1/2}).\]
\end{theorem}

Finally, Theorem \ref{cubflds} can also be used to
count the number $N(d,X)$ of cubic fields with absolute 
discriminant bounded by
$X$ and whose quadratic resolvent field is
$\mathbb{Q}(\sqrt{d})$.  To this end, let $M_3^D(X)$ be the number of cubic fields $K$ with shape of discriminant of $D$ and $|\Disc(K)| < X$.  Then by Theorem~\ref{cubflds}, we have $M_3^D(X) \sim \frac{1}{2}h(D)M_3^\Or(Q,X)$ as $X\to\infty$, for any primitive integral form $Q$ of discriminant $D$.      
Regarding $N(d,X)$,
we prove:

\begin{theorem}\label{NdtoM}
Suppose $d \neq -3$ is a fundamental discriminant.  Then
$$N(d,X) = \begin{cases}
M_3^{-3d}(X)& \text{if $3 \nmid d$},\\
M_3^{-3d}(X)+M_3^{-d/3}(X) & \text{if $3 \mid d$}.
\end{cases}$$
\end{theorem}
Combining this with Theorem \ref{cubflds}, we obtain the following.

\begin{theorem}\label{resolvents}
  Let $d \neq -3$ be a fundamental discriminant, and set $D =
  -\frac{d}{3}$ if $3\, |\, d$ and $D = -3d$ otherwise. Set $\alpha = 1$ if $3\,|\,D$ and $\alpha = 0$ otherwise. Also set $\beta = 1$
  for $D> -4 $ and $\beta = 0$ otherwise.  Then
$$N(d,X) = \frac{3^{\alpha + \beta  - \frac{1}{2}}\cdot C_0}{\pi^2
\sqrt{|D|}}\cdot \prod_{p | D}\frac{p}{p+1}\cdot L(D) \cdot X^{1/2}+o(X^{1/2})$$
where
$$L(D) = \prod_{p}\left( 1+ \left(\frac{D}{p}\right)\frac{1}{p+1}\right) = L(1,\chi_D) \prod_{\left(\frac{D}{p}\right)=1}\left(1-\frac{2}{p(p+1)}\right)$$
and
$$C_0 = \begin{cases}
11/9 & \text{if $d \not \equiv 0 \hspace{2mm}(\mod 3)$},\\
5/3 & \text{if $d \equiv 3 \hspace{2mm}(\mod 9)$},\\
7/5 & \text{if $d \equiv 6 \hspace{2mm}(\mod 9)$}.
\end{cases}$$
\end{theorem}
We note that this latter result on the asymptotic number of cubic
fields having a given quadratic resolvent field was recently obtained
independently by Cohen and Morra (\cite[Theorem~1.1(2),
Corollary~7.6]{CM}) using very different methods, and with an explicit
error term of $O(X^{1/3+\varepsilon})$.

All the above results may be extended to the case of square discriminant. The cubic fields with shape of square discriminant are precisely the \textit{pure} cubic fields---i.e. those of the form $\Q(\sqrt[3]{m})$ for some integer $m$---while the orders with such a shape are the orders in pure cubic fields (see Lemma \ref{purecub}).
The asymptotic growth in the case of a square discriminant is somewhat larger:

\begin{theorem}\label{pure}
Let $Q_D$ be a primitive integral binary quadratic form of square discriminant $D$. Set $\alpha = 1$ if $3\mid D$ and $\alpha = 0$ otherwise.  Then
$$N^\Or_3(Q_D,X) = \frac{ 3^{\alpha-\frac{3}{2}} }{2D}X^{1/2}\log X + \frac{3^{\alpha-\frac{3}{2}} }{D}\left(2\gamma-1+\frac{3}{2}\log\left(\frac{D}{3}\right)\right)X^{1/2}+ O(X^{1/4}),$$ where $\gamma$ is Euler's constant. Also, we have
$$M_3(Q_1,X) = \frac{C}{15\sqrt{3}}X^{1/2}\left(\log(X) - \frac{16}{5}\log(3) + 4\gamma + 12\kappa -2\right)+o(X^{1/2})$$ 
and $$M_3(Q_9,X) = \frac{C}{40\sqrt{3}}X^{1/2}\left(\log(X) - \frac{1}{5}\log(3) + 4\gamma + 12\kappa - 2\right)+o(X^{1/2})$$
where $C =\prod_p \left(1-\frac{3}{p^2}+\frac{2}{p^3}\right)$ and $\kappa = \sum_p \frac{\log p}{p^2+p-2}$.  For all other square values of $D,$ we have $M_3(Q_D,X) = 0$.  

Finally, $N(-3,X) = M_3(Q_1,X) + 2M_3(Q_9,X);$ hence the density of pure cubic fields is given 
by \begin{equation}\label{pureeq} N(-3,X) = \frac{7C}{60\sqrt{3}}X^{1/2}\left(\log(X)- 
\frac{67}{35}\log(3)+4\gamma+12\kappa - 2\right)+o(X^{1/2}). \end{equation}
\end{theorem} 
In particular, Theorem~\ref{pure} shows that, when counting cubic orders with shape of a given 
square discriminant $D$, both the first {\it and} {second} main terms of the asymptotics 
become {equidistributed} in the class group $\Cl_D$. 
Note that the last equation (\ref{pureeq}) in Theorem~\ref{pure} corresponds to 
Theorem 1.1(1) in Cohen and Morra's work~\cite{CM}, where they again prove an explicit error
term of $O(X^{1/3+\varepsilon})$.

We remark that there are two types of pure cubic fields: those with shape of discriminant $D = 1$ and those with shape of discriminant $D = 9$.  These turn out to correspond to Dedekind's notion of pure cubic fields of Types 1 and 2, respectively  (see \cite{Ded}).  In fact, the asymptotics for $N(-3,X)$ can be deduced 
fairly easily from Dedekind's work, as we explain in the last remark of the paper.  In general, Theorem~\ref{NdtoM} shows that there are two distinct types of cubic fields 
whenever the discriminant $d$ of the quadratic resolvent algebra is a multiple of~3.  In that case, there are cubic fields of shape of discriminant $-d/3$ and of discriminant $-3d$, and these are precisely the cubic fields of ``Type 1'' and ``Type 2'', respectively.  

Our methods are, for the most part, quite elementary, involving
primarily geometry-of-numbers arguments.  However, these methods also
fit into a larger context.  Let $G$ be an algebraic group and $V$ a
representation of $G$.  Then the pair $(G,V)$ is called a {\it prehomogeneous vector space} if $G$ possesses an open orbit; the pair $(G,V)$ is called a {\it
  coregular space} if the ring of invariants of $G$ on $V$ is free.  
  A classification of all {\it irreducible} reductive prehomogeneous vector spaces
was attained by Sato and Kimura in~\cite{SK}, while
  a similar classification of {\it simple} (resp.\ {\it semisimple} and {\it irreducible}) coregular
spaces was accomplished by Schwarz~\cite{Sch} (resp.\ Littelmann~\cite{Lit}), respectively.
The rational and integer orbits in such spaces tend to have a very rich arithmetic interpretation
(see, e.g., \cite{Gauss}, \cite{DF}, \cite{WY},
\cite{Bh1}--\cite{Bh4}, \cite{CFS}, \cite{BH}), and have led to several number-theoretic
applications, particularly to the study of the density of
discriminants of number fields 
and statistical questions involving elliptic and
higher genus curves (see \cite{DH}, \cite{DW},
\cite{Bh5}, \cite{Bh6}, \cite{BS}, and \cite{BG}).

In this paper, we prove Theorems 1--8 and related results by
considering the simplest nontrivial example of a coregular space whose
underlying group is {\it not} semisimple (namely, an action of
$\SO_Q(\C)$ on $\C^2$ for a binary quadratic form $Q$).  It also
yields the simplest prehomogeneous vector space that is {\it not}
irreducible (namely, an action of $\GO_Q(\C)$ on $\C^2$).  We show
that the integer orbits on this space, even in this non-irreducible
and non-semisimple scenario, also have a rich and nontrivial number-theoretic
interpretation, namely, the integer orbits classify cubic rings whose
lattice shape is $Q$.

In light of these results, we note that there has not been a
classification of general reducible prehomogeneous vector spaces nor of general non-semisimple coregular spaces, akin to the work of  Sato--Kimura or Schwarz--Littelmann,
respectively, leading to two interesting questions in
representation theory.  As we hope this paper will illustrate, the
solution of these two problems may also have important consequences for
number theory.

We note that the problems that we address in this paper are related to
problems considered by Terr~\cite{Terr} and more recently by
Mantilla~\cite{Mantilla} and Zhao~\cite{Zhao}.  In \cite{Terr}, Terr showed that the shape
of cubic rings (ordered by absolute discriminant) is equidistributed
amongst lattices, which are viewed as points in Gauss' fundamental
domain $\mathcal{F}$.  More precisely, the number of cubic rings (of
bounded discriminant) having shape lying in some subset $W \subset
\mathcal{F}$ is proportional to the area of $W$ (with respect to the
hyperbolic measure on $\mathcal{F}$).  Terr's work is somewhat
orthogonal to our own in that it implies that $N(Q,X) = o(X)$, but it
does not say anything more about a single shape $Q(x,y)$.  A related
problem is to determine when a cubic field is determined by its trace
form, which is a finer invariant than the shape.  This problem was
treated recently by Mantilla; see~\cite{Mantilla}.  
Finally, Zhao has recently carried out a detailed study of the distribution of cubic function
fields by discriminant, in which the geometric {\it Maroni invariant} of trigonal curves plays an important
role.  In particular, he has suggested an             analogue of the {Maroni invariant} for cubic number         
fields, which turns out to be closely related to the notion of ``shape''; see \cite{Zhao} for further details.          

\vspace{.15in}\noindent
\textit{Acknowledgments}.  We are very grateful to Hunter Brooks, Jordan Ellenberg, Wei Ho, Julian Rosen, Peter Sarnak, and Yongqiang Zhao for helpful conversations 
and comments on an earlier draft of this manuscript.  The first author was supported by a Packard Foundation Fellowship and National Science Foundation grant DMS-1001828.  
The second author was partially supported by National Science Foundation RTG grant DMS-0943832.

\section{Preliminaries}

In order to count cubic rings of bounded discriminant, we use a
parametrization, due to Delone-Faddeev~\cite{DF} and recently refined by
Gan-Gross-Savin~\cite{Gross}, that identifies cubic rings with
integral binary cubic forms.

\subsection{The Delone-Faddeev correspondence}

We follow the exposition of Gan-Gross-Savin~\cite{Gross}.  
Consider the space of all binary cubic forms $$f(x,y) = ax^3+bx^2y+cxy^2+dy^3$$
with integer coefficients, and let an element
$\gamma\in \GL_2(\Z)$ act
on this space by the twisted action 
$$\gamma\cdot f(x,y) = \frac{1}{\det(\gamma)} \cdot f((x,y)\gamma).$$ 
This
action is faithful and defines an equivalence relation on the space of
integral binary cubic forms.

The {\it discriminant} of a binary cubic form $f(x,y) =
ax^3+bx^2y+cxy^2+dy^3$ is defined~by
$$\Disc(f)=b^2c^2-4ac^3-4b^3d-27a^2d^2+18abcd.$$ 
The discriminant polynomial is invariant under the action of
$\GL_2(\mathbb{Z})$.  

\begin{theorem}[Gan-Gross-Savin~\cite{Gross}]\label{delone}
  There is a canonical bijection between isomorphism
  classes of cubic rings and 
  $\GL_2(\mathbb{Z})$-equivalence classes of integral binary cubic
  forms.  Under this bijection, the discriminant of a binary cubic form 
  is equal to the discriminant of the corresponding cubic ring. 
  Furthermore, a cubic ring is an integral domain $($i.e., a cubic
  order$)$ if and only if the corresponding binary cubic form is irreducible
  over $\mathbb{Q}$.
\end{theorem}

%
%

\begin{proof}
See \cite[\S4]{Gross} and \cite[\S2]{BST}.
\end{proof}

\vspace{-.125in}
\begin{remark}
{\em  To parametrize \textit{oriented} cubic rings, one must use
  $\SL_2(\mathbb{Z})$-equivalence in the correspondence of
  Theorem~\ref{delone}, rather than
  $\GL_2(\mathbb{Z})$-equivalence.  Recall that the shape of an
  oriented cubic ring is then well-defined up to
  $\SL_2(\mathbb{Z})$-equivalence.}
\end{remark}

Not only does Theorem~\ref{delone} give a bijection between cubic
rings and cubic forms, but it also shows that certain properties and
invariants of each type of object translate nicely.  
Next, we describe how
the shape of a cubic ring 
translates into the language of binary cubic forms.  

\subsection{Hessians and shapes}
Let $f(x,y) = ax^3+bx^2y+cxy^2+dy^3$ be an
integral binary cubic form.
The integral binary quadratic form
$$H_f(x,y) = -\frac{1}{4}\det\left(\begin{array}{cc} \frac{\partial f(x,y)}{\partial x^2} & \frac{\partial f(x,y)}{\partial x\partial y} \\ \frac{\partial f(x,y)}{\partial y \partial x}& \frac{\partial f(x,y)}{\partial y^2}\end{array}\right)$$
is called the \textit{Hessian} of $f$.  The Hessian
has the following properties \cite{Gross}.

\begin{proposition}\label{hess}
  The Hessian is a $\GL_2(\mathbb{Z})$-covariant of integral binary
  cubic forms; that is, if two binary cubic forms $f$ and $g$ are
  equivalent under the twisted action of $\GL_2(\mathbb{Z})$, then the
  corresponding Hessians $H_f$ and $H_g$ are also equivalent under the
  same action but without the twisting factor $($i.e., the factor of
  the determinant$)$.  For any binary cubic form $f$, we have
  $\Disc(H_f)=-3\cdot \Disc(f)$.
\end{proposition}

The relevance of the Hessian of a binary cubic form is that it gives
the shape of the corresponding cubic ring:

\begin{proposition}\label{corr}
  Suppose that a cubic ring $($resp.\ oriented cubic ring$)$ $R$
  corresponds to a binary cubic form $f(x,y)$ as in
  Theorem~$\ref{delone}$.  Then the $\GL_2(\Z)$ $($resp.\ $\SL_2(\Z))$-equivalence class
  of the primitive part of the Hessian $H_f(x,y)$ coincides with the
  shape of $R$.
\end{proposition}

\begin{proof}
If we write $$\gamma = x\alpha +y\beta = \frac{\Tr(\gamma)}{3} + \gamma_0,$$ 
with $\gamma_0 \in \frac{1}{3}R$ of trace zero, then a computation gives $H_f(x,y) = \frac{3}{2}\Tr(\gamma_0^2)$ (see \cite{Gross}).
\end{proof}

\vspace{-.125in}
\begin{example}
 {\em Suppose $R$ is a cubic ring having an order three automorphism and 
  corresponding binary cubic form $f(x,y)$.  Then the Hessian $H_f(x,y)$ of
  $f$ must have an order three automorphism as well, and so it must be equivalent to an integer multiple of the quadratic form $Q(x,y) = x^2+xy+y^2$.
  Conversely, we show in the next section that any ring having
  the form $Q(x,y)$ as its shape must be a $C_3$-cubic
  ring.}
\end{example}

\section{On cubic orders having automorphism group $C_3$}
In this section we will prove Theorems \ref{c3ordercount} and
\ref{cohnthm}.  As mentioned in the introduction, these theorems are
actually special cases of Theorems \ref{cubrings} and \ref{cubflds},
respectively.  We prove these cases separately because
the argument is better motivated and understood after
seeing a concrete example.  Moreover, the results in this case are interesting in
their own right due to their connection with $C_3$-cubic orders in abelian cubic fields.

\subsection{The action of $\SO_{Q}(\C)$ on $\C^2$}

Set $Q(x,y) = x^2+xy+y^2$, and let $\SO_{Q}(\C)$ denote the subgroup of elements of
$\SL_2(\C)$ preserving the quadratic form $Q(x,y)$ via
its natural (left) action on binary quadratic forms; i.e., 
\[\SO_{Q}(\C)=\left\{\gamma\in\SL_2(\C):Q(x,y)=Q((x,y)\gamma)\right\}.\]
We define the {\it cubic action} of $\SO_{Q}(\mathbb{C})$ on
$\mathbb{C}^2$ by $\gamma \cdot v = \gamma^3v$ for a column vector $v=(b,c)^t \in
\mathbb{C}^2$.  The adjoint quadratic form
$Q'(b,c):=b^2-bc+c^2$ of $Q(x,y)$ is an invariant polynomial for this latter action, 
and it generates the full ring of invariants.

Let $L\subset \C^2$ be the lattice $\{(b,c)^t:b,c\in\Z^2,\,\,b\equiv
c\pmod{3}\}$.  We will see that 
$L$ is preserved under $\SO_{Q}(\Z)$, the group
of integer matrices in $\SO_{Q}(\C)$.

\begin{theorem}\label{Qorb}
  The $\SO_{Q}(\mathbb{Z})$-orbits on nonzero lattice vectors $(b,c)^t\in L$ are in
  natural bijection with $C_3$-cubic oriented rings $R$.  Under this bijection, 
  $\Disc(R)=Q'(b,c)^2.$
\end{theorem}

\begin{proof} 
 Let $R$ be a $C_3$-cubic ring with automorphism $\sigma$ of order
  3, and let $$f(x,y)=ax^3+bx^2y+cxy^2+dy^3$$ be a binary cubic form corresponding to $R$ under
  the Delone-Faddeev correspondence.
Also let
  $H(x,y)$ denote the Hessian of $f(x,y)$.  Then $\sigma$
  induces an order 3 automorphism on $f$ and hence on $H$.  Up to $\SL_2(\Z)$ equivalence and scaling, there is only one integral binary quadratic form having an $\SL_2(\Z)$-automorphism of order three, namely $Q(x,y) = x^2+xy+y^2$.  Thus, after a change of basis, we may assume that
  $H(x,y)=nQ(x,y)$ for some nonzero $n\in\Z$.  Hence we have
  \begin{equation}\label{Heq1}
    b^2-3ac=n,\,\,bc-9ad=n,\,\,c^2-3bd=n. 
  \end{equation}   
  The first equation implies $a=(b^2-n)/3c$, while the third implies 
  $d=(c^2-n)/3b$ (assuming $b,c$ nonzero).  
  Substituting these values of $a,d$ into the second equation gives 
  \[b^2c^2-(b^2-n)(c^2-n)=nbc.\] Expanding out, and dividing by $n$,
  we obtain $$b^2-bc+c^2=Q'(b,c) = n.$$  We now have
\begin{equation}\label{adsol}
a=\frac{bc-c^2}{3c} = \frac{b-c}{3} \,\,\mbox{ and }\,\, d=\frac{bc-b^2}{3b}
= \frac{c-b}{3};
\end{equation} 
 one easily checks that this gives the unique solution for $a$ and
 $d$ even when one of $b$ or $c$ is zero.
Furthermore, $f$ has integer coefficients precisely when $(b,c)^t\in L$,
i.e., $b$, $c$ are integers
congruent modulo~3.  

Conversely, if $(b,c)^t \in L$ is nonzero, then we can define integers
$a$ and $d$ as in (\ref{adsol}) and set $
f(x,y)=ax^3+bx^2y+cxy^2+dy^3$; this cubic form has Hessian $H_f(x,y)
= nQ(x,y)$ where $n = Q'(b,c).$ A calculation shows that the
order 3 transformation
\begin{equation}\label{S}S =  \left(\begin{array}{cc}-1& 1\\ -1& 0\end{array}\right)\end{equation} 
is an automorphism of $f(x,y)$, and hence the ring $R$ corresponding to
$f$ is a $C_3$-cubic ring.

Now suppose we have binary cubic forms $f$, $f'$ corresponding to $(b,c)^t$,
$(b',c')^t \in L$.  Then $f$ and $f'$ are $\SL_2(\Z)$-equivalent if and
only if they are $\SO_{Q}(\Z)$-equivalent, since they both have Hessian
equal to $nQ$.  Write $f(x,y)=((b-c)/3)x^3+bx^2y+cxy^2+((c-b)/3)y^3$ and $f'=\gamma f$ with
$\gamma\in\SO_Q(\Z)$.  Then 
an elementary computation shows that 
$$(\gamma f)(x,y)=f((x,y)\gamma) = ((b'-c')/3)x^3+b'x^2y+c'xy^2+((c'-b')/3)y^3,$$
where $(b',c')^t = \gamma^3 (b,c)^t$.  (Note that the cubic action here is to be expected because the cube roots of the identity generated by $S\in\SO_Q(\Z)$ must lie in the kernel of the action of $\SO_Q(\Z)$ on $L$.)
It follows that $f$ and
$f'$ are $\SO_{Q}(\mathbb{Z})$-equivalent if and only if $(b,c)^t$ and
$(b',c')^t$ are $\SO_{Q}(\mathbb{Z})$-equivalent under the cubic
action.  This proves the first part of the theorem.  
Finally, by Propositions~\ref{hess} and \ref{corr} we know that
$\Disc(R)=\Disc(f)=-\frac13\Disc(H)$, and we compute
$-\frac13\Disc(H)=n^2=Q'(b,c)^2$.
\end{proof}



\subsection{The number of $C_3$-cubic orders of bounded discriminant}

To prove Theorem~\ref{c3ordercount}, we first require the following
lemma which shows that the reducible forms $f(x,y)$ corresponding to
$C_3$-cubic rings are negligible in number.  This will allow us to prove
asymptotics for $C_3$-cubic \textit{orders} rather than just $C_3$-cubic rings.

\begin{lemma}\label{c3irred}
  The number of $\SL_2(\Z)$-equivalence classes of reducible integral
  binary cubic forms having Hessian a multiple of $Q(x,y) =
  x^2+xy+y^2$, and discriminant less than~$X$, is $O(X^{1/4})$.
\end{lemma}  


\begin{proof}
We give an upper estimate for the number of $\SL_2(\Z)$-equivalence classes of reducible forms $f$  of discriminant less than $X$ whose Hessian is a multiple of $Q$.  It will suffice to count first the primitive forms $f$, and then we will sum over all possible contents for $f$.
Now any such primitive reducible $f$ with Hessian $nQ$ has a linear factor $gx+hy$ with $g$ and $h$ relatively prime integers.  Furthermore, the order three automorphism (\ref{S}) permutes
    the three roots of $f$ in $\P^1$, and hence $f$ must factor into the 3 primitive linear factors
  that are obtained by successively applying $S$ to $gx+hy$.  Thus we have
  $$f(x,y) = (gx+hy)((h-g)x-gy)(-hx+(g-h)y).$$ 
  Computing the discriminant of $f$, we find
  $$\Disc(f)= (g^2 - gh + h^2)^6 = Q'(g,h)^6.$$
  Thus, if $\Disc(f)<X$, then $Q'(g,h)<X^{1/6}$, and hence the total number of values for the pair $(g,h)$, and thus $f$, is at most $O(X^{1/6})$.  
  
  In order to count the total number of forms $f$ and not just the primitive ones, we sum over all possible values of the content $c$ of $f$.  Since $\Disc(f/c)=\Disc(f)/c^4$, we thus obtain
  \begin{equation}\label{contentsum}
  \sum_{c=1}^{X^{1/4}} O((X/c^4)^{1/6}) = O(X^{1/4})
  \end{equation}
  as an upper estimate for the total number of $\SL_2(\Z)$-equivalence classes of reducible forms $f$  of discriminant less than $X$ whose Hessian is a multiple of $Q$, as desired.
 \end{proof}

{\noindent {\bf Proof of Theorem 1: }} By Theorem~\ref{Qorb} and
Lemma~\ref{c3irred}, it now suffices to count elements $(b,c)^t\in L$, up
to $\SO_{Q}(\Z)$-equivalence, subject to the condition
$Q'(b,c)^2=(b^2-bc+c^2)^2<X$.  The number of integral points inside
the elliptic region cut out by the latter inequality is approximately
equal to its area $(2\pi/\sqrt3)X^{1/2}$, with an error of at most
$O(X^{1/4})$ \cite{Cohn}.  Meanwhile, being the (orientation-preserving) symmetry
group of the triangular lattice, $\SO_{Q}(\Z)$ is isomorphic to
$C_6$, the cyclic group of order 6.
Since this is the cubic action, the cyclic subgroup
$C_3\subset \SO_{Q}(\Z)$ of order 3 acts trivially.  Up to equivalence, we thus obtain
$$\frac{2\pi}{2\sqrt3}X^{1/2}+O(X^{1/4})$$ points inside the ellipse.  The number of such points with $b\equiv c$ (mod 3) is therefore
$$ \frac{\pi}{3\sqrt3}X^{1/2}+O(X^{1/4}).$$ This is the
number of oriented $C_3$-cubic rings with discriminant bounded by $X$. By Lemma~\ref{c3irred}, the $C_3$-cubic rings that are not orders will be
absorbed by the error term.  After dividing by two to account for the fact that we counted oriented rings, we obtain the formula in Theorem~\ref{c3ordercount}.
{$\Box$ \vspace{2 ex}}

\subsection{An elementary proof of Cohn's theorem on the number of abelian cubic fields of 
bounded discriminant}

Now we wish to count those points $(b,c)^t\in L$ of bounded discriminant
corresponding to maximal cubic orders.
This is equivalent to counting abelian cubic extensions of $\Q$ of
bounded discriminant.  Since maximality is a local property, it
suffices to determine how many $C_3$-cubic rings $R$ satisfy the
condition that the $\Z_p$-algebra 
$R \otimes \mathbb{Z}_p$ is maximal for every prime $p$.
The following lemma (\cite[Lemma 13]{BST}) gives a useful criterion
to determine when a cubic ring $R$ is maximal at~$p$.
\begin{lemma}\label{maxcrit}
  If $f$ is a binary cubic form over $\Z$ $($or $\Z_p)$, then $R(f)$
 is not maximal at $p$ if and only if one of the following conditions hold:
\begin{itemize}
\item $f(x,y) \equiv 0$ $(\mod p)$ 
\item $f$ is $\GL_2(\mathbb{Z})$-equivalent to a form $f'(x,y) = ax^3+bx^2y+cxy^2+dy^3$ such that $a \equiv 0$  $(\mod p^2)$ and $b \equiv 0$ $(\mod p)$.  
\end{itemize}
\end{lemma}
In particular, if $\Disc(f)$ is nonzero $($mod $p^2)$ then $R(f)$ is maximal at~$p$. 

With the help of this lemma, we now determine conditions for when the
cubic order $R(f)$
corresponding to the binary cubic form
\begin{equation}\label{f}
 f(x,y)=\Bigl(\frac{b-c}3\Bigr) x^3+ b x^2 y + c x y^2 + \Bigl(
\frac{c-b}3\Bigr)y^3
\end{equation}
is maximal at~$p$.  We consider three cases, corresponding to the
three possible residue classes of $p$~($\mod 3$).

First suppose that $p \equiv 2$  ($\mod 3$).  We have $\Disc(f) = Q'(b,c)^2$ by Theorem \ref{Qorb}, and $Q'(x,y) = x^2-xy+y^2$ does not factor (mod $p$).  Then, by the lemma, $R(f)$ is
maximal at~$p$ as long as $(b,c)\not\equiv (0,0)$ (mod
$p$).  We conclude that for $p \equiv 2$ $(\mod p$), the $p$-adic density of elements
$(b,c)^t\in L$ corresponding to maximal rings (at $p$) is $1-\frac{1}{p^2}$.

Next, suppose $p\equiv 1$ (mod 3).  In order for $Q'(b,c)$ to vanish modulo $p$, we
need $c\equiv \zeta b$ (mod $p$), where $\zeta$ is a primitive
sixth root of unity in $\Z/p\Z$.  In that case, we obtain
\begin{equation}\label{fex} f(x,y)\equiv 
\frac b3\zeta^{-1} (x^3+3\zeta x^2y + 3\zeta^2xy^2 + \zeta^3 y^3)
\equiv \frac b3\zeta^{-1} (x + \zeta y)^3 \pmod{p}.
\end{equation}
If $b\equiv 0$ (mod $p$), then $f(x,y)\equiv 0$ (mod~$p$) and so $R$
is not maximal at $p$.  Otherwise, if we have a pair $(b,c)$ with
$c\equiv \zeta b$ (mod p) and $b\not\equiv 0$ (mod $p$), then we may send
the unique multiple root of $f(x,y)$ in $\P_{\F_p}^1$ to the point
$(0,1)$ via a transformation in $\GL_2(\Z)$.  Then, modulo~$p$, the form $f(x,y)$ is congruent to a multiple of $y^3$.  A proportion of
$\frac{1}{p}$ of these forms satisfy $a \equiv 0$ (mod $p^2$), where
$a$ is the coefficient of $x^3$.  By Lemma \ref{maxcrit}, the $p$-adic
density of points $(b,c)^t\in L$ corrresponding to rings maximal at $p$
is therefore
\[1 - \frac{1+\frac{2(p-1)}{p}}{p^2}
= \frac{p^3 - 3p +2}{p^3} = \frac{(p-1)^2(p+2)}{p^3}\] since there are two primitive sixth roots of unity $\zeta$ in
$\Z/p\Z$ if $p\equiv 1$~(mod 3).  

Finally, if $p=3$, then we wish to know the density of all $(b,c)^t\in L$
for which the binary cubic form $f(x,y)$ in (\ref{f}) yields 
a cubic ring maximal at~3.  Clearly, we need $b\equiv -c$ (mod~3) for
the discriminant $(b^2-bc+c^2)^2$ of the corresponding cubic ring to vanish
modulo~3.  Since already $b\equiv c$ (mod 3), we must have $b\equiv
c\equiv 0$ (mod 3) to obtain a ring that is not maximal at 3.  Write
$b=3B$ and $c=3C$.  Then we wish to know when the binary cubic
\[g(x,y)=(B-C)x^3+3Bx^2y+3Cxy^2+(C-B)y^3\] corresponds to a cubic ring
not maximal at 3.  Note that $g(x,y)\equiv (B-C)(x-y)^3$ (mod~3), and
$g(x,y)\equiv 0$ (mod~3) if and only if $B\equiv C$ (mod~3).
Otherwise, we can send the unique multiple root of $g(x,y)$ in
$\P_{\F_p}^1$ to $(0,1)$, which transforms $g(x,y)$ to a multiple of
$y^3$.  As before, a proportion of $\frac{1}{3}$ of such forms will
have $x^3$ coefficient congruent to 0 (mod $p^2$).  By
Lemma~\ref{maxcrit}, it follows that a proportion of
$\frac{1}{3}\left(\frac{1}{3}+\frac{2}{3}\cdot\frac{1}{3}\right) =
\frac{5}{27}$ of such $g(x,y)$ will correspond to a cubic ring not
maximal at 3.  We conclude that the density of $(b,c)^t\in L$ that yield
a cubic ring maximal at 3 is 22/27.

We have proven the following proposition.

\begin{proposition}
Let $S_{\max}$ denote the set of all $(b,c)^t\in L$ corresponding to rings maximal at $p$ under the bijection of Theorem~$\ref{Qorb}$. Then the $p$-adic density $\mu_p(S_{\max})$ of $S_{\max}$ in $L$ is given by
\begin{equation}
\mu_p(S_{\max}) = \left\{
\begin{array}{ll}
\frac{(p-1)^2(p+2)}{p^3} & \mbox{if $p\equiv 1$ {\em (mod 3)}} \\
1-\frac{1}{p^2} & \mbox{if $p\equiv 2$ {\em (mod 3)}} \\
\frac{22}{27} & \mbox{if $p=3$} 
\end{array}
\right. .
\end{equation}
\end{proposition}

The proof of Theorem 1 gives the total number $N(L;X)$ of points in $L$, up to $\SO_{Q}(\Z)$-equivalence, having discriminant at most $X$.  We may similarly determine the number $N(S;X)$ of points, up to $\SO_{Q}(\Z)$-equivalence, having discriminant at most $X$, where $S$ is any $\SO_Q(\Z)$-invariant subset of $L$ defined by finitely many congruence conditions.

\begin{proposition}
Let $S\subset L$ be an $\SO_Q(\Z)$-invariant subset that is defined by congruence conditions modulo finitely many prime powers.  Then the number $N(S;X)$ of points in $S$, up to $\SO_{Q}(\Z)$-equivalence, having discriminant at most $X$ is given by
\begin{equation}\label{cong} N(S;X) = \frac{\pi}{6\sqrt3} \prod_p \mu_p(S)\cdot X^{1/2}+O_S(X^{1/4}), \end{equation}
where $\mu_p(S)$ denotes the $p$-adic density of $S$ in $L$.
\end{proposition}
The proposition follows from arguments essentially identical to those in the proof of Theorem~1.

The set $S_{\max}$ of elements in $L$ that correspond to maximal cubic rings, however, is defined by infinitely many congruence conditions.  To show that (\ref{cong}) still holds for such a set, we require a uniform estimate on the error in (\ref{cong}) when the congruence conditions defining $S_{\max}$ are imposed only at the finitely many primes $\leq Y$, as $Y\to\infty$.  This is the content of the following

\begin{proposition}\label{unif}
Let $S_{\max}^{\leq Y}$ denote the subset of $L$ corresponding to cubic rings maximal at all primes $\leq Y$.  Then
$$N(S_{\max}^{\leq Y};X) - N(S_{\max};X) = O(X^{1/2}/Y).$$
\end{proposition}

\begin{proof}
Let $W_p$ denote the subset of elements in $L$ corresponding to cubic rings $R$ that are not maximal at $p$.  Any such ring $R$ is contained in a maximal ring $R'$, where $R'$ also has shape $x^2+xy+y^2$ (this is because the field containing $R$ must have an order 3 automorphism, and then so does $R'$).  The number of such possible $R'$ (up to isomorphism) with discriminant less than $X$ is $O(X^{1/2})$ by Theorem~1. To count all orders $R$ in such $R'$ having discriminant less than $X$, we require the following lemma.

\begin{lemma}[Datskovsky--Wright~\cite{DW}]
If $R'$ is any maximal cubic ring, then the number of orders $R\subset R'$ of index $m=\prod p_i^{e_i}$ is $O_\epsilon(\prod p_i^{(1+\epsilon)\lfloor e_i/3\rfloor})$ for any $\epsilon>0$.
\end{lemma}

The lemma implies that the total number of cubic rings $R$ of discriminant less than $X$ that are not maximal at $p$ and are contained in maximal rings $R'$ of shape $x^2+xy+y^2$ is at most
$$
\Bigl(\sum_{e=1}^\infty
\frac{p^{(1+\epsilon)\lfloor e/3\rfloor}}{p^{2e}}\Bigr)
\prod_{q\neq p}
\Bigl(\sum_{e=0}^\infty
\frac{q^{(1+\epsilon)\lfloor e/3\rfloor}}{q^{2e}}\Bigr) O(X^{1/2})
= O(X^{1/2}/p^2).$$
Since $\sum_{p\geq Y} O(X^{1/2}/p^2)=O(X^{1/2}/Y)$, we obtain the desired estimate.
\end{proof}

Thus, by choosing $Y$ large enough, we can make $N(S_{\max}^{\leq Y};X) - N(S_{\max};X) 
\leq cX^{1/2}$
for any $c>0$.  We conclude that the number of $C_3$-cubic fields of discriminant less
than $X$ is asymptotic to 
\[\frac{\pi}{6\sqrt3}\cdot\frac{22}{27}\prod_{p\equiv1(3)}
\frac{(p-1)^2(p+2)}{p^3}\prod_{p\equiv 2(3)}
\bigl(1-\frac1{p^2}\bigr)\cdot X^{1/2}
=\frac{11\pi}{3^4\sqrt{3}}\cdot \frac6{\pi^2}\cdot\frac{9}{8}\prod_{p\equiv1(3)}
\frac{(p-1)(p+2)}{p(p+1)}\cdot X^{1/2}
\]
which is
\[ \frac{11\sqrt{3}}{36\pi}\prod_{p\equiv1(3)}\Bigl(1 -
\frac2{p(p+1)}\Bigr)\cdot X^{1/2}\]
and this is the result of Cohn~\cite{cohnab}.

\section{On cubic orders having a general fixed lattice shape}

Let $Q(x,y)$ be a primitive integral binary quadratic form with non-square discriminant.  In this section we determine asymptotics for the number $N_3^\Or(Q,X)$
of oriented cubic orders having absolute discriminant bounded by $X$
and shape $Q$, i.e., we prove Theorem~\ref{orcubrings}.  To accomplish this, we generalize
the proofs of the previous section.  

We choose to work with \textit{oriented} cubic rings for a couple of reasons.  First, this allows us to
ignore the determinant $-1$ automorphisms in $\GO_Q(\mathbb{Z})$,
making the proof a bit simpler.  Second, Theorem~\ref{orcubrings} shows that, at least asymptotically, lattice shapes are equidistributed within
the (narrow) class group,
suggesting that oriented rings are the natural framework for our analysis.   

\subsection{A more general action of $\SO_Q(\C)$ on $\C^2$}

Recall that the shape of a cubic ring $R$ is an equivalence
class of binary quadratic forms.~We begin by fixing a representative
of this class, say $Q(x,y) = rx^2+sxy+ty^2$.  As in the $C_3$-cubic
ring case, we consider a lattice $L=L(Q)$ of elements $(b,c)^t \in \Z^2$, defined by the
congruence conditions
$$ sb \equiv rc \hspace{1mm} (\mod \hspace{1mm} 3t)
\mbox{ and } sc \equiv tb \hspace{1mm} (\mod \hspace{1mm}3r).$$ 
Let $Q'(x,y) = tx^2-sxy+ry^2$ denote the adjoint quadratic form of $Q(x,y)$.
Let 
$\SO_{Q}(\C)$ denote the subgroup of transformations $\gamma\in\SL_2(\Z)$ which fix $Q$
via $(\gamma Q)(x,y)=Q((x,y)\gamma)$. 
Then we define the {\it cubic action} of $\SO_Q(\C)$ on $\C^2$
just as before, namely $\gamma \cdot v = \gamma^3v$ for $v \in \C^2$.
We will see that $\SO_Q(\Z)$, the subgroup of elements in $\SO_Q(\C)$ having 
integer entries, preserves $L$, and the quadratic form $Q'$ gives an invariant polynomial on $L$.  Define the subset 
$$L(Q)^+ =  \left\{(b,c)^t \in L(Q) : \frac{Q'(b,c)}{rt} > 0\right\} \subset L(Q).$$  Then we have the following generalization of Theorem~\ref{Qorb}.    

\begin{theorem}\label{bij}
  Let $Q(x,y) = rx^2+sxy+ty^2$ be a primitive integral quadratic form with non-square
  discriminant, and let $Q'(x,y) = tx^2-sxy+ry^2$ denote the adjoint quadratic form of $Q(x,y)$.  
  Then the orbits of the
  cubic action of $\SO_{Q}(\Z)$ on lattice points $(b,c)^t \in L(Q)^+$
  are in natural bijection with the isomorphism classes of oriented
  cubic rings $R$ having shape $Q$.  Under this
  bijection, we have $$\Disc(R) =
  -\frac{Q'(b,c)^2\Disc(Q)}{3r^2t^2}.$$
\end{theorem}     

\begin{proof}  
  The proof is similar to that of Theorem~\ref{Qorb}.  Let
  $R$ be a cubic ring with shape~$Q$.  Then, by applying an appropriate
  $\SL_2(\mathbb{Z})$-transformation, we may assume that the corresponding
  integral cubic form $f(x,y) = ax^3+bx^2y+cxy^2+dy^3$ has Hessian
  $$H(x,y) = n(rx^2+sxy+ty^2) = nQ(x,y),$$
  with $n$ positive.  From the definition of the
  Hessian, we have 
\begin{equation}\label{rr}
b^2-3ac = nr,\hspace{3mm} bc-9ad =ns,
  \hspace{3mm}c^2-3bd = nt,
\end{equation} for some positive integer $n$.  Assuming $b,c$ are nonzero, these
  equations imply that
$$a=(b^2-nr)/3c \hspace{2mm} \mbox{and} \hspace{2mm} d = (c^2-nt)/3b.$$
Using the middle equation in (\ref{rr}), we find that
\begin{equation} tb^2 - sbc + rc^2 = ntr. \label{equ} \end{equation}
We then get
\begin{equation} a = \frac{sb-rc}{3t}, \hspace{3mm} d =
  \frac{sc-tb}{3r}\label{coeff} \end{equation} (note that $r$ and $t$ are
nonzero because $\Disc(Q)$ is not a square), and one checks that this
is the unique solution even if $b$ or $c$ is zero.  Notice that
$f(x,y)$ has integer coefficients if and only if
\begin{equation} sb \equiv rc \hspace{1mm} (\mod \hspace{1mm} 3t)
  \mbox{ and } sc \equiv tb \hspace{1mm} (\mod \hspace{1mm}3r),
  \label{congs} \end{equation} i.e., $(b,c)^t \in L$.  In this case, we even have $(b,c)^t \in L(Q)^+$, since $n$ is positive.  We see that the form $f(x,y)$ is
determined once we specify the shape $Q$ and the middle coefficients
$b$ and~$c$.  Conversely, given any element $(b,c)^t \in L(Q)^+$, we may use the
equations in (\ref{coeff}) to define a cubic form $f = (a,b,c,d)$ such that $R(f)$ has
shape $Q$ and the Hessian of $f$ is $nQ$ for some positive integer $n$.

Now suppose $f = (a,b,c,d)$ and $f' = (a',b',c',d')$ are two
binary cubic forms chosen such that the Hessians are $nQ$ and $n'Q$
with integers $n,n' > 0$; thus, the respective conditions in~(\ref{coeff}) and (\ref{congs}) hold for the coefficients of $f$ and
$f'$.  Then $f$ and $f'$ are $\SL_2(\mathbb{Z})$-equivalent if
and only if they are $\SO_Q(\Z)$-equivalent.  A computation as in the proof of
Theorem~\ref{Qorb} 
shows that 
if $f'=\gamma f$ for $\gamma\in\SO_Q(\Z)$, then  
$(b',c')^t = \gamma^3 (b,c)^t$.  
It follows that $f$ and~$f'$ are $\SO_{Q}(\mathbb{Z})$-equivalent if and only if $(b,c)^t$ and
$(b',c')^t$ are $\SO_{Q}(\mathbb{Z})$-equivalent under the cubic action.


Thus we have proved the bijection described in the theorem.  Furthermore,
we have
$$\Disc(R) = \Disc(f) = -\frac{n^2\Disc(Q)}{3}$$ 
by Proposition \ref{hess}, and combining with $Q'(b,c) = nrt$, which
was Equation (\ref{equ}), we obtain the desired result.
\end{proof}
   
\begin{remark}
{\em If $Q$ is positive definite, then $L(Q)^+$ is the set of nonzero vectors in $L(Q)$.  If~$Q$ is negative definite then $L(Q)^+$ is empty.  If $Q$ is indefinite, then nonzero elements of $L$ not in $L(Q)^+$ correspond to cubic rings with shape $-Q$. }   
\end{remark}

\subsection{The number of cubic orders of bounded discriminant and
  given lattice shape}

We are nearly ready to prove Theorem~\ref{orcubrings}.  
We consider the cases of definite and indefinite $Q$ separately.  

\subsubsection{Definite case}
In this case, we have the following well known lemma.
\begin{lemma}\label{groupsize}
Let $Q(x,y)$ be a definite integral binary quadratic form.  The order
of $\SO_{Q}(\mathbb{Z})$ is either $6$, $4$, or $2$ depending on whether
the form $Q$ $($up to equivalence and scaling$)$ is $x^2+xy+y^2$,
$x^2+y^2$, or any other definite form. 
\end{lemma}

We next prove the analogue of Lemma~\ref{c3irred} for general definite forms. 

\begin{lemma}\label{reduc}
Let $Q(x,y)$ be a definite integral binary quadratic form of non-square discriminant $D$.  Then the number of $\SL_2(\Z)$-equivalence classes of reducible cubic forms $f$ having shape $Q$ and $|\Disc(f)| < X$ is $O(X^{1/4})$.  
\end{lemma}

\begin{proof}
We  give an upper estimate for the number of $\SL_2(\Z)$-equivalence classes of forms~$f$ of discriminant less than $X$ whose Hessian is a multiple of the fixed definite binary quadratic form $Q(x,y) = rx^2+sxy+ty^2$ of discriminant $D$.  
As in the proof of Lemma~\ref{c3irred}, it suffices to first count just those $f$ that are primitive.  

Now an order three automorphism of such a primitive $f$ (over $\bar \Q$) of determinant 1 is given by 

\begin{equation*}S =  \left(\begin{array}{cc}-\frac{-\sqrt{D} - s\sqrt{-3}}{2\sqrt{D}}& r\sqrt{\frac{-3}{D}}\\  -t\sqrt{\frac{-3}{D}}& -\frac{-\sqrt{D} + s\sqrt{-3}}{2\sqrt{D}}\end{array}\right).\end{equation*} 
The transformation $S$ permutes the roots of $f$ in $\P^1(\bar\Q)$.  So if $gx+hy$ is a linear factor of $f(x,y)$ with $g$ and $h$ coprime integers, then by computing the two other linear factors (obtained by applying $S$ and $S^{-1}$), and multiplying all three factors together, we obtain a polynomial $f_0$ which must agree with $f$ up to scaling.  Since $\sqrt{D}S$ is an integral matrix of discriminant $D$, the content of $f_0$ must divide $D^2$.  Computing the discriminant of $f_0$, we obtain
$$\Disc(f_0) = -\frac{27}{D^6}(tg^2 - sgh + rh^2)^6 = -\frac{27}{D^6}Q'(g,h)^6.$$ 
Since $Q'$ is definite and $\Disc(f_0)\leq D^8\Disc(f) <D^8 X =O(X)$, we see that the total number of choices for $(g,h)$, and thus $f$, is $O(X^{1/6})$.  

Finally, to obtain an upper estimate for the count of all $f$ that are not necessarily primitive, we sum over all possible contents $c$ of $f$ as in Equation (\ref{contentsum}).  We obtain a total of $O(X^{1/4})$ possibilities for $f$, as desired.
\end{proof}

We denote by $C(Q)$ the cardinality of the subgroup of cubes in $\SO_{Q}(\Z)$.  Thus $C(Q) = |\SO_{Q}|$ unless $Q$ has discriminant $-3$ (which is the $C_3$-case already dealt with in~\S3).  

\begin{theorem}\label{numorders}
  Let $Q(x,y) = rx^2+sxy+ty^2$ be a positive definite primitive
  integral quadratic form of discriminant $-D$.  Let
  $\alpha = 1$ if $3\mid D$ and $\alpha = 2$ otherwise.
  Then $$N_3^\Or(Q,X) =
  \frac{2\pi \sqrt{3}}{3^\alpha C(Q)D}X^{1/2}+O(X^{1/4}). $$ 
\end{theorem}

\begin{proof}
  By Theorem \ref{bij} and Lemma \ref{reduc}, it suffices to count equivalence classes of elements $(b,c)^t$ in $L(Q)^+$ such that
  $$\frac{Q'(b,c)^2D}{3r^2t^2} < X$$ or equivalently $$tb^2 -
  sbc+rc^2 < \frac{\sqrt{3}rtX^{1/2}}{\sqrt{D}}.$$ In this case, $L(Q)^+$ is the simply the set of nonzero vectors in $L(Q)$.  The number of $L(Q)$-points in the elliptic region in $\mathbb{R}^2$ defined by
  the inequality above is approximately given by the area of this
  ellipse\footnote{The area bounded by an ellipse with equation $Q(x,y) =
    N$ is $\frac{2\pi N}{\sqrt{D}}$.}  divided by the area $\Vol(L)$
  of a fundamental parallelogram of $L$. The error is at
  most\footnote{In fact, work on Gauss' circle problem gives exponents considerably smaller than
    1/4; see \cite{hardy}.}  $O(X^{1/4})$, where the implied constant
  depends on the shape of $L$ and thus on $Q$ \cite{Cohn}.  So we have
\begin{equation}\label{asym}N_3^Q(X) = \frac{2\pi rt \sqrt{3}X^{1/2}}{\Vol(L)D}+ O(X^{1/4}).\end{equation}
We further divide by $C(Q)$ to obtain the number of points in $L$ satisfying the
inequality up to the cubic action equivalence.  Thus the theorem
follows from the following lemma.  

\begin{lemma}\label{vol}
  Let $r,s,t$ be integers with no common prime factor and set $D =
  s^2-4rt$.  Also set $\alpha = 1$ if $3\mid D$ and $\alpha = 2$
  otherwise.  Then the lattice $$L(Q) = \{(b,c)^t \in \Z^2: sb \equiv rc
  \hspace{1mm} (\mod \hspace{1mm} 3t) \mbox{ and } sc \equiv tb
  \hspace{1mm} (\mod \hspace{1mm}3r)\}$$ has volume $3^\alpha rt.$
\end{lemma}

\begin{proof}
The following proof is due to Julian Rosen.  Let $L'$ be the lattice in $\Z^2$ generated by $(3t,0)^t$ and $(0,3r)^t$.  Then $L$ is the inverse image of $L'$ under the map $\Z^2 \to \Z^2$ which sends $(x,y)^t$ to $(sx-ry,tx-sy)^t$.  Then $\Vol(L) = \frac{\Vol(L')}{\Vol(C)}$, where $C$ is the lattice spanned by columns of the matrix 
$$M = \left(\begin{array}{cccc}s& -r & 3t & 0\\ t& -s & 0 & 3r\end{array}\right).$$
Since $\Vol(C)$ is the greatest common divisor of the two-by-two minors of $M$, which is precisely $3^{2-\alpha}$, we have $\Vol(L) =  \frac{9rt}{3^{2-\alpha}} =  3^\alpha rt$ as claimed. 
\end{proof}

\noindent
This concludes the proof of Theorem~\ref{numorders}.
\end{proof}   

\vspace{-.125in}

\begin{remark}
{\em It is natural to sum the main term of Theorem \ref{numorders}
  over all shapes of negative discriminant to try to recover
  Davenport's count~\cite{D} of the number of cubic rings with
  positive discriminant bounded by $X$.  The method of \cite{Siegel} makes it possible to compute this sum, but it turns out not to equal Davenport's main term $\frac{\pi^2}{72}X$.  Evidently, the error term in Theorem \ref{numorders} is contributing to the main term of this sum}.

\end{remark}


\subsubsection{Indefinite case}

Next we consider counting cubic orders having shape $Q(x,y) =
rx^2+sxy+ty^2$, where $Q(x,y)$ is indefinite.  The
equations (\ref{coeff}) are not well-defined if either $r$ or $t$ is zero, so we assume that $D = \Disc(Q) = s^2-4rt$ is not a square.  We may also assume that $t > 0$.

As in
the proof of Theorem~\ref{numorders}, we need to count elements $(b,c)^t$
in $L(Q)^+$ up to $\SO_{Q}(\mathbb{Z})$-equivalence, such that
$$\left|\frac{Q'(b,c)^2D}{3r^2t^2}\right| < X$$ or in other words
$$\left|tb^2 - sbc+rc^2\right| <
\left|\frac{\sqrt{3}rtX^{1/2}}{\sqrt{D}}\right|.$$ This inequality
cuts out a region in the plane bounded by a hyperbola, and we need to
count the orbits of the {cubic} action of $\SO_{Q}(\mathbb{Z})$ on $L(Q)^+$
which intersect this region.  But $\SO_Q(\Z)$ is now an infinite group, so we must construct a fundamental domain for the action at hand
(the construction and the ensuing volume computation are taken from
\cite[Chapter 6]{Dav}).

In what follows, it is useful to define $\theta = \frac{s +
  \sqrt{D}}{2t}$ and $\theta' = \frac{s - \sqrt{D}}{2t}$.  We then
have $$Q'(x,y) = tx^2-sxy+ry^2= t(x-\theta y)(x-\theta' y).$$  We also have the following well known facts.  

\begin{proposition}\label{sols}
\hspace{1mm}
\begin{itemize}
\item
The integral solutions $(U,W)$ of ``Pell's equation'' $u^2 - Dw^2 = 4$ are given by
  $$\frac{1}{2}\left(U+W\sqrt{D}\right) = \pm\left[\frac{1}{2}\left(U_0+W_0\sqrt{D}\right)\right]^n$$
  where $n$ is any integer and $(U_0,W_0)$ is a minimal solution.  
\item Every element $M$ in $\SO_{Q}(\Z)$ is of the form 
$$M = \left(\begin{array}{cc}\frac{1}{2}(U+sW)& -rW\\ tW& \frac{1}{2}(U-sW)\end{array}\right)$$
for some solution $(U,W)$ to Pell's equation.  
\end{itemize}
\end{proposition}
If $(X,Y)^t \in \mathbb{Z}^2$ and $M\cdot (X,Y)^t= (x,y)^t$ for some $M \in \SO_{Q}(\Z)$, then the second part of the proposition implies $$\frac{x-\theta'y}{x-\theta y} =
\frac{\frac{1}{2}(U+W\sqrt{D})}{\frac{1}{2}(U-W\sqrt{D})}\cdot
\frac{X-\theta'Y}{X-\theta Y}$$ for some solution $(U,W)$ to Pell's
equation.  If we define $\epsilon = \frac{1}{2}(U_0+W_0\sqrt{D}) > 1$,
then the first part gives $$\frac{1}{2}(U+W\sqrt{D}) =
\pm\epsilon^m \mbox{ and }  \hspace{1mm}\frac{1}{2}(U-W\sqrt{D}) = \pm\epsilon^{-m}$$ for some
integer $m$.  Thus, in each orbit of $L(Q)^+$ there is a single element $(x,y)^t$
such that $$1 \leq \frac{x-\theta'y}{x-\theta y} < \epsilon^2$$ and
$x - \theta y > 0$.

Our goal, then, is to count the number of integer points $(x,y)^t$ obeying the following constraints:
$$tx^2-sxy+ry^2 \leq N = \left|\frac{\sqrt{3}rtX^{1/2}}{\sqrt{\Disc(Q)}}\right|, \hspace{5mm} x - \theta y > 0, \hspace{5mm} 1 \leq \frac{x-\theta'y}{x-\theta y} < \epsilon^2.$$  This region is a sector emanating from the origin bounded by a hyperbola.  Just as in the positive definite case, we can approximate this count by computing the area of this region.  

Changing coordinates from $x,y$ to $$\xi = x - \theta y,
\hspace{5mm} \eta = x-\theta'y,$$ this region is $$\xi \eta \leq N/t,
\hspace{5mm} \xi > 0, \hspace{5mm} \xi \leq \eta < \epsilon^2\xi.$$
These conditions can be rewritten as $$0 < \xi \leq
(N/t)^{\frac{1}{2}}, \hspace{5mm} \xi \leq \eta <
\min\left(\epsilon^2\xi,\frac{N}{t\xi}\right).$$ If we define $\xi_1 =
\epsilon^{-1}(N/t)^{\frac{1}{2}}$, then the area of this region is
$$\int_0^{\xi_1} (\epsilon^2\xi -
\xi)d\xi+\int_{\xi_1}^{(N/t)^{\frac{1}{2}}}
\left(\frac{N}{t\xi}-\xi\right)d\xi.$$ Evaluating the integrals we obtain
$$(\epsilon^2-1)\frac{1}{2}\xi_1^2+\frac{N}{2t}\log(N/t)-(N/t)\log\xi_1
- \frac{1}{2}(N/t)+\frac{1}{2}\xi_1^2$$ which simplifies to $(N/t)\log
\epsilon$.  This is the area in the $\xi,\eta$-coordinate
system; to get the area in $x,y$-coordinates we divide by $$\frac{\partial(\xi,\eta)}{\partial(x,y)} =
\theta -\theta' = \frac{\sqrt{D}}{t}.$$

Thus the desired area is $(N/\sqrt{D})\log \epsilon$.  However, recall
that we are interested in the orbits of the cubic action of $\SO_{Q}$
on the lattice $L(Q)$, so we must replace
$\epsilon$ by $\epsilon^3$ in the previous calculation.  In the final area estimate, this yields
an extra factor of 3 (since $\log \epsilon^3 = 3\log \epsilon$).

We may now continue as in the rest of the proof of
Theorem~\ref{numorders}.  Note that the proof of Lemma~\ref{reduc} carries over to the indefinite case because we can again bound the number of points of discriminant $<X$ in the fundamental domain in terms of the corresponding area.  The final result is 

\begin{theorem}\label{indef} 
  Let $Q(x,y) = rx^2+sxy+ty^2$ be a primitive integral quadratic form
  having non-square discriminant $D > 0$.  Let $\alpha=1$
  if $3\mid D$ and $\alpha=2$ otherwise.  Then $$N_3^\Or(Q,X) =
  \frac{3\sqrt{3}\log \epsilon}{3^\alpha D}X^{1/2}+O(X^{1/4}).$$
\end{theorem}
\vspace{.2in}\noindent
Dirichlet's class number formula \cite{Dav} states that $$h(D) =
\frac{w\sqrt{|D|}}{2\pi}L(1,\chi_D)$$ for $D < 0$, where $w$ is the
number of roots of unity in $\Q(\sqrt{D})$, and $$h(D) =
\frac{\sqrt{D}}{\log \epsilon}L(1,\chi_D)$$ for $D>0$.\footnote{Recall that $h(D)$ is the \textit{narrow} class number.} Using these
equations, we see that Theorems \ref{numorders} and \ref{indef} can be
combined and the result is Theorem \ref{orcubrings}.

\subsection{The number of maximal cubic orders of bounded discriminant and
  given lattice shape}

As before, let $Q(x,y) = rx^2+sxy+ty^2$ be a quadratic form over
$\mathbb{Z}$, with $D = s^2-4rt$ not a square, and let us further assume
that $Q$ is primitive.

In this subsection, we use the results of \S4.2 to provide asymptotics
for $M_3^\Or(Q,X)$, the number of isomorphism classes of {\it maximal}
oriented cubic orders having shape $Q$.  We may ignore those
maximal cubic rings that are not orders, because such rings can be
written as $\mathbb{Z} \oplus S$ where $S$ is a maximal quadratic ring
of discriminant (a rational square multiple of) $-3\Disc(Q)$.  But there
is at most one such maximal quadratic ring and this one exception will be
absorbed by the error term.  Thus $M_3(Q,X)$ (resp.\ $M_3^\Or(Q,X)$) is
essentially the number of cubic fields (resp.\ oriented cubic
fields) with ring of integers of shape $Q$ and absolute discriminant less than $X$.

Just as in the $C_3$ case in Section 3, we compute the $p$-adic
density of those elements in $L=L(Q)$ corresponding to cubic rings
of shape $Q$ that are maximal at $p$.  For every such ring, we can choose
a corresponding integral binary cubic form to be
$$f(x,y)=\frac{sb-rc}{3t}x^3 +bx^2y+cxy^2+\frac{sc-tb}{3r}y^3$$ 
for some pair $(b,c)^t$ in the lattice $L$ defined by the congruence
conditions $sb\equiv rc$ ($\mod 3t$) and $sc\equiv tb$ ($\mod 3r$).
We proved in the previous section that 
\begin{equation}\label{dsc}\Disc(f) =\frac{D}{3r^2t^2}Q'(b,c)^2,\end{equation} 
where $Q'(x,y) = tx^2-sxy+ry^2$.  

As it differs from other primes, we first consider primes other
than $p=3$, and then treat $p=3$ separately.  The reader only
interested in the results may consult Table \ref{table:maxdens}
at the end of the section. 

In what follows, we denote the reduction (mod $p$) of the form
$Q(x,y)$ by $Q_p(x,y)$.

\subsubsection{The $p$-adic density for maximality ($p\neq 3$)}\label{pdensity}

We naturally divide into three cases.

\vspace{.2in}\noindent
{\bf Case 1: $Q_p(x,y)$ has distinct roots in $\mathbb{F}_{p}$}
\vspace{.12in}  

Using an $\SL_2(\Z)$-transformation, we may arrange for $Q_p(x,y)$ to
be $sxy$.  The congruence conditions defining $L$ imply that
$b \equiv c \equiv 0$ (mod $p$) in this case.  Then $$f(x,y) \equiv
Ax^3+ Dy^3\hspace{2mm} (\mod p)$$ where $A$ and $D$ are independent
parameters in $\Z$.  Since $\Disc(f) \equiv -27A^2D^2$ (mod $p$), the
cubic ring
$R(f)$ is maximal unless either $A$ or $D$ is 0 (mod $p$).  By
Lemma \ref{maxcrit}, $R(f)$ is not maximal at $p$ precisely when
either one of $A$ or $D$ is 0 (mod $p^2$) or they simultaneously
vanish (mod $p$).  Thus the $p$-adic density of the set of $(b,c)^t \in
L$ that give rise to maximal rings at $p$ is $$1-\frac{1}{p^2}
-\frac{2(p-1)}{p^3} = \frac{(p-1)^2(p+2)}{p^3}.$$

\vspace{.21in}\noindent
{\bf Case 2: $Q_p(x,y)$ has distinct roots in $\mathbb{F}_{p^2}-\F_p$}
\vspace{.12in}  

We note that the discriminant $D$, as well as the outer coefficients $r$ and
$t$, must be nonzero $(\mod p$) in this case.  Since $Q_p$ (and
hence $Q'_p$) does not factor modulo $p$, we see from (\ref{dsc}) that
$p$ divides $\Disc(f)$ if and only if $b \equiv c \equiv 0$ (mod $p$),
in which case $f(x,y)$
vanishes (mod~$p$) and so $R(f)$ is not maximal at $p$.  Thus 
the $p$-adic density of the set of $(b,c)^t \in L$
that give rise to maximal rings at $p$ in this case is $1-\frac{1}{p^2}$.

\vspace{.2in}\noindent
{\bf Case 3: $Q_p(x,y)$ has a double root in $\mathbb{F}_{p}$}
\vspace{.12in}  

The quadratic form $Q_p$ has a double root in $\F_p$ if and only if
$p$ divides the discriminant~$D$ of $Q$.  By sending the double root
(mod~$p$) of $Q_p$ to 0 via a transformation in $\SL_2(\mathbb{Z})$,
we may assume that $Q_p$ is the form $rx^2$ (mod $p$).

Since $t \equiv 0$ ($\mod p$), we see that $c \equiv 0$ ($\mod p$) for
all $(b,c)^t \in L$, by the definition of~$L$.  Thus in $\F_p$, we have
$$f(x,y) \equiv \frac{sb-rc}{3t}x^3+bx^2y \equiv
x^2\left(\frac{sb-rc}{3t}x+by\right).$$ The coefficient
$\frac{sc-tb}{3r}$ of $y^3$ in $f(x,y)$ is 0 (mod $p^2$) precisely
when $p^2$ divides $tb$.  If $p > 3$, then $p^2$ divides $t$ if and only if
it divides $D = s^2 - 4rt$.  Thus for $p > 3$, the $p$-adic density of the set of
$(b,c)^t \in L$ that give rise to maximal rings at $p$ is 0 if $p^2 \mid D$
and $1 - \frac{1}{p}$ if $p \parallel D$.  If $p = 2$, then we write $D = 4m$ and note that $4\mid t$ if and
only if $m$ is congruent to 0 or 1 (mod 4).
Thus the density is $0$ when $m$ is congruent to 0 or 1 (mod 4)
 and is $\frac{1}{2}$ when $m$ is
congruent to 2 or 3 (mod 4).

\subsubsection{The 3-adic density for maximality}

We again divide into three cases.   

\vspace{.2in}\noindent
{\bf Case 1: $Q_3(x,y)$ has distinct roots in $\mathbb{F}_{3}$}
\vspace{.12in}  

In this case, it is most convenient to assume that $Q_3(x,y) = x(x+y)$.  The congruence
conditions defining $L$ imply that $b\equiv c \equiv 0$ (mod 3).  We
then find that $$f(x,y) \equiv -Nx^3-My^3 =
-(NX+MY)^3$$ for parameters $N,M \in \Z$.  After sending the single
root of $f(x,y)$ over $\F_3$ to 0, we see that aside from the
degenerate form, one third of these forms will
have the coefficient of $y^3$ congruent to 0 (mod $3^2$).  By Lemma~\ref{maxcrit}, the $3$-adic density of the set of $(b,c)^t \in L$ that
give rise to maximal rings at $3$ is
$$1-\frac{1+\frac{1}{3}(3^2-1)}{3^2} = \frac{16}{27}.$$

\vspace{.2in}\noindent
{\bf Case 2: $Q_3(x,y)$ has distinct roots in $\mathbb{F}_{9}-\F_3$}
\vspace{.12in}  

Now we may assume that $Q_3(x,y)$ is
the form $x^2+y^2$.  Then the conditions defining the
lattice $L$ imply that $b \equiv c \equiv 0$ ($\mod 3$) for all
$(b,c)^t \in L$.  So over $\F_3$ we have $$f(x,y) \equiv -Cx^3-By^3 =
-(CX+BY)^3$$ where $b = 3B$ and $c = 3C$.  Arguing as in the previous case, we conclude that the 3-adic density of the set of $(b,c)^t \in L$ that give rise to maximal rings at 3 is $\frac{16}{27}$.

\vspace{.2in}\noindent
{\bf Case 3: $Q_3(x,y)$ has a double root in $\mathbb{F}_{3}$}
\vspace{.12in}  

We may assume that $Q_3(x,y) = rx^2$, so $c \equiv 0$ (mod 3) for all
$(b,c)^t \in L$.  We first consider the case where $3 || D$, which
implies that $3 || t$.  Notice that $$\Disc(f) =
\frac{D}{3r^2t^2}Q'(b,c)^2$$ is divisible by 3 if and only if $b
\equiv 0$ (mod 3).  But if $b \equiv 0$ (mod 3), then $c \equiv 0$
(mod 9), $$f(x,y) \equiv \frac{sb-rc}{3t}x^3 \hspace{2mm} (\mod 3)$$
and the coefficient of $y^3$ vanishes modulo $p^2$ precisely when $b
\equiv 0$ (mod 9).  Thus the $3$-adic density of the set of $(b,c)^t \in
L$ that give rise to maximal rings at $3$ is $$1 -
\frac{1}{3}\left(\frac{1}{3}+\frac{2}{3}\cdot\frac{1}{3}\right) =
\frac{22}{27}.$$

Next we assume that 9 divides $D$, so that 9 divides $t$ as well.  By definition, $sb \equiv rc$ (mod 27) for all $(b,c)^t \in L$; in particular $c \equiv 0$ (mod 3).  We also have that $$f(x,y) \equiv \frac{sb-rc}{3t}x^3+bx^2y \hspace{2mm} (\mod 3).$$  By Lemma \ref{maxcrit}, it suffices to determine when the coefficient $\frac{sc-tb}{3r}$ vanishes modulo $9$, i.e., when $sc-tb$ vanishes modulo 27.  This happens when 3 divides $b$ (because then 9 divides $c$).  If 3 does not divide $b$, since $sb \equiv rc$ (mod 27) the congruence $sc \equiv tb$ (mod 27) is equivalent to $s^2 \equiv rt$ (mod 27) which is equivalent to $27 \mid D$.  Thus if $27 \mid D$, then there are no maximal rings $R(f)$, and if $9 || D$, then the $3$-adic density of the set of $(b,c)^t \in L$ that
give rise to maximal rings at $3$ is $\frac{2}{3}$.  
 
Below is a table showing the $p$-adic density of points $(b,c)^t \in L$
which give rise to a maximal ring at $p$.  In practice, one determines
whether a particular prime fits into Case 1, 2, or 3 based on whether
the quadratic residue $\left(\frac{D}{p}\right)$ is $-1$, 1, or 0, 
respectively.  For $p=2$, this is not well defined, but for convenience
we define
$$\left(\frac{D}{2}\right)=
\begin{cases}
\;1& \text{if $D\equiv 1 \hspace{1mm}(\mod 8)$},\\
-1& \text{if $D\equiv 5 \hspace{1mm} (\mod 8)$}.
\end{cases}$$ 
Note that these densities, which we denote by $\mu_p(D)$, are in fact a function of
the discriminant $D$ of $Q$, i.e. they are independent of the
particular equivalence class of $Q$.

\begin{table}[ht]
\renewcommand{\arraystretch}{1.25}
\centering 
\begin{tabular}{c c c c c c c}  
\hline\hline  
& $\left(\frac{D}{p}\right) = -1$ & $\left(\frac{D}{p}\right) = 1$ & $p \| D$ & $p^2 \| D$ & $p^3 \| D$ & $p^4 |D$ \\ [1ex] 
\hline   
$p=2$ & $\frac{3}{4}$ & $\frac{1}{2}$ & - & $\frac{1}{2} \mbox{ or } 0$ &$\frac{1}{2}$ & 0\\[1ex]
$p=3$ & $\frac{16}{27}$ & $\frac{16}{27}$ & $\frac{22}{27}$ & $\frac{2}{3}$ & 0 & 0\\ [1ex] 
$p \geq 5$ & $1 - \frac{1}{p^2}$ & $\frac{(p-1)^2(p+2)}{p^3}$ & $1 - \frac{1}{p}$ & 0 & 0 & 0\\ [1ex] 
\hline 
\end{tabular}
\caption{Densities $\mu_p(D)$ of $(b,c)^t \in L$ corresponding to rings maximal at $p$} 
\label{table:maxdens} 
\end{table}

In this table, there is an ambiguity in the case where $p=2$ and $4
\|D$.  To resolve the ambiguity, one writes $D = 4m$, and if $m \equiv 1$ (mod 4) then the density is 0,
whereas if $m \equiv 3$ (mod 4) then the density is $\frac{1}{2}$.  The table implies the following proposition.

\begin{proposition}\label{fund}
  If $R(f)$ is a maximal cubic ring with shape $Q(x,y)$ of
  discriminant $D$, then either $D$ is a fundamental discriminant or
  $-D/3$ is a fundamental discriminant.
\end{proposition}


We can now use Theorem \ref{numorders}, together with the analogue of the uniformity estimate of Proposition \ref{unif} (the proof carries over to any $L$ via Proposition \ref{fund}) to compute the number of maximal cubic orders
(equivalently, cubic fields) of bounded discriminant and with shape $Q(x,y)$.  The uniformity estimate allows us (as in \S 3.3) to multiply each $p$-adic density of rings maximal
at $p$, for a given shape $Q(x,y)$, to obtain the proportion of maximal
cubic orders of that shape.  By the previous proposition, we need only
consider quadratic forms with discriminant $D$ squarefree away from 2
or 3.

Recall that $M_3^\Or(Q,X)$ denotes the number of oriented
maximal cubic orders with shape $Q$ and absolute discriminant less
than $X$, and $N_3^\Or(Q,X)$ denotes the number of oriented cubic
orders with shape $Q$ and absolute discriminant less than $X$.
\begin{theorem}\label{maxorders}
Let $Q(x,y)$ be a quadratic form whose discriminant is not a square.
Then $$M_3^\Or(Q,X) = N_3^\Or(Q,X) \prod_p \mu_p(D)+o(X^{1/2}).$$
\end{theorem}

As a corollary, we prove Theorem~\ref{cubflds}.

\vspace{.1in}

\begin{proof1}
Using Theorem \ref{orcubrings}, we compute the main term of $M_3^\Or(Q,X)$
(in the following products over primes, $p=3$ is never included): 
$$ \frac{3^{\alpha + \beta - \frac{3}{2}}L(1,\chi_D)}{h(D) \sqrt{|D|}}\cdot
  \mu_3(D)\prod_{\left(\frac{D}{p}\right)=-1}\left(1-\frac{1}{p^2}\right)\prod_{\left(\frac{D}{p}\right)=1}\left(\frac{(p-1)^2(p+2)}{p^3}\right)\prod_{p | D}\left(1-\frac{1}{p_i}\right)X^{1/2}$$
$$=\frac{3^{\alpha + \beta - \frac{3}{2} }L(1,\chi_D)}{h(D) \sqrt{|D|}}\cdot\mu_3(D)\cdot\frac{6}{\pi^2}\cdot\frac{9}{8} \prod_{\left(\frac{D}{p}\right)=1}\left(1-\frac{2}{p(p+1)}\right)\prod_{p | D}\left(\frac{p_i-1}{p_i}\cdot\frac{p_i^2}{p_i^2-1}\right)X^{1/2}$$ 
$$= \frac{3^{\alpha + \beta + \frac{1}{2} }L(1,\chi_D)}{4\pi^2h(D) \sqrt{|D|}}\mu_3(D)\prod_{\left(\frac{D}{p}\right)=1}\left(1-\frac{2}{p(p+1)}\right)\prod_{p | D}\left(\frac{p_i}{p_i+1}\right)X^{1/2}$$ with the first equality following from the fact that $\zeta(2) = {\pi^2}/{6}$.  
\end{proof1}

\section{On cubic fields having a given quadratic resolvent field}

Suppose $K$ is a cubic field whose ring of integers $R_K$ has shape
$Q(x,y)$ and $\Disc(Q) = D$ is not a square.  If $f$ is an integral
cubic form corresponding to the cubic ring $R_K$, then $K =
\mathbb{Q}(\theta)$ where $\theta$ is a root of $f(x,1)$.  Since
$\sqrt{\Disc(f)}$ is the product of differences of the roots of $f$,
we see that the field $\mathbb{Q}(\sqrt{\Disc(f)})$ is contained in
the Galois closure of $K$.  Unless $D = -3$ (which is when $K$ is
Galois), this field will be quadratic.  The field
$\mathbb{Q}(\sqrt{\Disc(f)})$ is called the \textit{quadratic
  resolvent field} of $K$.

Now suppose $d$ is a fundamental discriminant.  Proposition \ref{fund} and the equation 
$$\Disc(f) = \frac{-Dn^2}{3}$$ (for some integer $n$) shows that $K$
will have $\mathbb{Q}(\sqrt{d})$ as its quadratic resolvent if and
only if the shape of $R_K$ has discriminant
$$D = \begin{cases}
-3d& \text{if $d\not \equiv 0 \hspace{2mm} (\mod 3)$},\\
-3d \text{ or } \frac{-d}{3} & \text{if $d \equiv 0 \hspace{2mm}(\mod 3)$}.
\end{cases}$$
          
Define $N(d,X)$ to be the number of cubic fields with absolute
discriminant bounded by $X$ and with quadratic resolvent field
$\mathbb{Q}(\sqrt{d})$.  Also define $M^D_3(X)$ to be the number of cubic
fields whose rings of integers have shape of discriminant $D$ and
whose discriminant is less than $X$. Then we have
\begin{equation}N(d,X) = M_3^{-3d}(X)\label{notmult}\end{equation}
if $d$ is not a multiple of 3 and 
\begin{equation}N(d,X) =
  M_3^{-3d}(X)+M_3^{-d/3}(X)\label{mult3}\end{equation} if $d$ is a
multiple of 3.  Notice that the result of Cohn proved earlier
gives the asymptotics for $N(1,X)$, the number of abelian cubic
extensions of $\Q$ of bounded discriminant.  To generalize this result to general $d$, we require:

\begin{proposition}\label{class}
Let $Q$ be a primitive integral binary quadratic form of non-square discriminant $D$.  Then 
$$M^D_3(X) = \frac{1}{2}h(D)M_3^{\Or}(Q,X) + o(X^{1/2}).$$
\end{proposition}

\begin{proof}
Let $H(D)$ be the set of primitive integral binary quadratic forms of discriminant $D$.  If $Q \in H(D)$, set $\gamma(Q) = 1$ if $Q$ is an ambiguous form and set $\gamma(Q) = 0$ otherwise.  We have
\begin{align*}
M^D_3(X) &= \sum_{Q\in H(D)/\GL_2(\Z)} M_3(Q,X)+o(X^{1/2})\\
 &= \sum_{Q\in H(D)/\GL_2(\Z)} 2^{-\gamma(Q)}M_3^{\Or}(Q,X)+o(X^{1/2})\\
 &= \sum_{Q\in H(D)/\SL_2(\Z)} \frac{1}{2} M_3^{\Or}(Q,X) +o(X^{1/2})\\
&= \frac{1}{2}h(D)M^\Or_3(Q_0,X) +o(X^{1/2})
\end{align*}
where $Q$ varies over representatives of the equivalence classes in the sums and $Q_0$ is an arbitrary element of $H(D)$. 
\end{proof}

\begin{proof2}
  If $d$ is not a multiple of 3, then we can estimate $N(d,X)$ by combining Theorem \ref{cubflds} (with $D = -3d$), Proposition \ref{class}, and Equation
  (\ref{notmult}).   In this case, we have $\alpha = 1$ and $\mu_3(D)
  = \frac{22}{27}$.  We need only check that the constants in the resulting main term agree with those in Theorem  \ref{resolvents}.  Plugging in $\alpha = 1$, $\mu_3(-3d) = 22/27$, and $C_0 = 11/9$ we immediately see that they do.  

  If $d$ is a multiple of 3, we use Equation (\ref{mult3}) instead of Equation (\ref{notmult}).  The extra summand in Equation (\ref{mult3}) makes the calculation slightly more involved.  The key is to write all of the constants
  in the main term of $M_3^{-3d}(X)$ as a function of the discriminant
  $D = -\frac{d}{3}$.  If we define $D_1 = -3d = 9D$,
  then our table of $p$-adic densities of maximalities gives
  $\mu_3(D_1) = \frac{2}{3}$ and $\mu_3(D) = \frac{16}{27}$.  Also, one uses the fact that
  $$L(D_1) = \begin{cases}
\frac{4}{3}L(D)& \text{if $d \equiv 3 \hspace{2mm}(\mod 9)$},\\
\frac{4}{5}L(D)& \text{if $d \equiv 6 \hspace{2mm}(\mod 9)$}.
\end{cases}$$
After grouping the common factors, one then checks that the remaining numerical constant is equal to the value of $C_0$ stated in Theorem \ref{resolvents}.
\end{proof2}

\section{On cubic orders with shape of square discriminant}
\subsection{Pure cubic rings}

We begin by describing when a cubic order has shape of square discriminant.
\begin{lemma}\label{purecub}
A cubic order has shape of square discriminant if and only if it is contained in a pure cubic field, i.e., a field of the form $\Q(\sqrt[3]{m})$ for some $m \in \Z$.   
\end{lemma}

\begin{proof}
Although this is well known, we include a proof here as we could not find one in the literature.  Let $R$ be a cubic order with shape of square discriminant $D$.  Then $\Disc(R) = -n^2D/3$ for some integer $n$.  Since the discriminant of an order $R$ differs from the discriminant of the maximal order by a square factor, it suffices to prove the lemma for maximal orders $R = \O_K$.  Note that $D$ is a square if and only if the quadratic resolvent field of $K$ is $F = \Q(\sqrt{-3}) = \Q(\mu_3)$.  If $K$ is a pure cubic field, then certainly its quadratic resolvent field is $F$.  

Conversely, suppose $K$ has quadratic resolvent equal to $F$, and write $M$ for the Galois closure of $K/\Q$.  By Kummer theory, $M = F(\theta)$, where $\theta$ is a root of $x^3 - \alpha =0$, and where $\alpha =  a + b\sqrt{-3} \in \O_F$ is cubefree\footnote{Since $\O_F$ is a UFD, this notion makes sense (up to roots of unity). }.  We are finished if $b = 0$ or if $N_{M/K}(\theta) \notin \Z$, so assume $b \neq 0$ and $N_{M/K}(\theta) = n \in \Z$.  In this case, $N_{F/\Q}(\alpha) = a^2 + 3b^2 =  n^3$.  Note that the six conjugates of $\theta$ are the six cube roots of $a\pm b \sqrt{-3}$.  Let $\theta' $ be the conjugate of $\theta$ satisfying $\theta\theta'  = n$.  Then $\theta + \theta'$ is in $K$ and has minimal polynomial $g = x^3 - 3nx - 2a$.  But $\Disc(g) = 324b^2$ is a square, so $K$ is Galois over $\Q$, a contradiction.       
\end{proof}

We consider a primitive integral quadratic form $Q(x,y) = rx^2+sxy+ty^2$ with discriminant $D = s^2-4rt = m^2$, a square in $\Z$.  The goal is to estimate the number $N^\Or_3(Q,X)$ of oriented cubic orders having discriminant bounded by $X$ and having shape $\tilde{Q}$, the $\SL_2(\Z)$-equivalence class of $Q$.  It is not difficult to show that we may take $Q(x,y)$ of the form  $rx^2+sxy$, with $0\leq r<s = \sqrt{D}$.  A representative form $Q$ of this type will be called \textit{reduced}.  Note that the primitivity of $Q$ implies that $(r,s) = (r,D) = 1$.      

Recall that there is a bijective correspondence between isomorphism classes of oriented cubic rings having shape $Q$ and $\SL_2(\Z)$-equivalence classes of integral binary cubic forms $f(x,y) = ax^3+bx^2y+cxy^2+dy^3$ whose Hessian $H(x,y)$
is equal to $nQ(x,y)$ for some positive integer $n$,
i.e., the cubic form $f$ satisfies the following equations:

\begin{equation}\label{Heq}
    b^2-3ac=nr,\,\,bc-9ad=ns,\,\,c^2-3bd=0. 
  \end{equation}
If $r = 0$, then the primitivity of $Q$ forces $s = 1 = D$.  In this case, $b = c = 0$, so the space of such cubic forms is precisely those of the form $f(x) = ax^3+dy^3$, with $a$ and $d$ nonzero.  For square $D \neq 1$ we may assume that $r \neq 0$, and one checks that all the variables involved are nonzero.  Combining the equations in (\ref{Heq}), we find that $3nrd = nsc$ and so $$c = \frac{3rd}{s}, \hspace{3mm} b = \frac{3r^2d}{s^2}=\frac{cr}{s}.$$ 
We see that such forms are determined in a linear fashion by the outer coefficients $a$ and $d$.  Using the equations (\ref{Heq}) once more, we may write $n$ in terms of the other variables: $$n = \frac{9d}{D^2}(r^3d-s^3a).$$  We obtain the following formula for the discriminant of the associated cubic ring:
$$\Disc(R) = \Disc(R(f(x,y))) = -\frac{\Disc(H(x,y))}{3} = -\frac{Dn^2}{3} = -\frac{27}{D^3}(r^3d^2-s^3ad)^2$$
  The cubic form 
$$f = ax^3 + \frac{3r^2d}{s^2}x^2y + \frac{3rd}{s}xy^2 + dy^3$$ is integral if and only if $(a,d)^t \in \Z^2$ lies in the lattice 
$$L(Q) =  \left\{(a,d)^t  \in \Z^2 : 3d \equiv 0 \hspace{1mm} (\mod D)\right\}.$$  Thus, the set of oriented cubic rings having shape $Q$ is parametrized by the set $L(Q)^+ \subset L(Q)$ of elements $(a,d)^t$ such that $n = {9d}(r^3d-s^3a)/D^2 > 0$, modulo $\SO_Q(\Z)$-equivalence. Here, $\SO_Q(\Z)$ is the subgroup of $ \gamma \in \SL_2(\Z)$ such that $\gamma \cdot Q = Q$.  A simple computation shows that $|\SO_Q(\Z)| = 2$ when $D$ is a square.

Lemma \ref{reduc} continues to hold when $Q$ has square discriminant (the proof is similar), so we can again safely ignore the contribution coming from reducible cubic forms when counting cubic rings having shape $Q$.  Thus the main term of $N^\Or_3(Q,X)$ is obtained by counting the number of lattice points $(a,d)^t\in L(Q)^+$ (modulo $\SO_Q(\Z)$-equivalence) such that $$0 < |r^3d^2-s^3ad| < \frac{D^{3/2}X^{1/2}}{3\sqrt{3}} = N$$ and $3d \equiv 0$ $ (\mod D)$.  For $s = 1$ (and hence $r = 0$), this problem amounts to counting lattice points under a rectangular hyperbola, for which there is Dirichlet's estimate 
$$ 2N\log N +2(2\gamma-1)N+O(N^{1/2})$$ where $\gamma$ is Euler's constant.  For $s>1$, we use a similar estimate to find $$\frac{2N}{s^3} \log N+\frac{2N}{s^3}\left(2\gamma-1\right)+O(N^{1/2})$$ such lattice points.  We then divide by the size of $\SO_Q(\Z)$ and the covolume of the lattice of points $(a,d)^t \in \Z^2$ such that $3d \equiv 0$ $(\mod D)$.  If we set $\alpha = 1$ if $3\mid D$ and $\alpha = 0$ otherwise, then this volume is ${D}/{3^\alpha}$.  Putting this all together, we obtain the first part of Theorem \ref{pure}.

\subsection{Pure cubic fields}

As we did with shapes of non-square discriminant, we would like to take the product of local maximality densities at each prime $p$ to obtain a formula for the number $M_3(Q,X)$ of maximal cubic orders of shape $Q$ and discriminant bounded by $X$.  However, the sieve we used does not work as well in this case, because the fundamental domain for cubic forms with shape of square discriminant is not convex.  Instead, we describe the lattice points which give rise to maximal orders directly.

\begin{proposition}\label{puremax}
If $R$ is a maximal cubic ring with shape $Q(x,y)$ and $D = \Disc(Q)$ is a square, then either $D = 1$ or $D = 9$.  
\end{proposition}
\begin{proof}
If $D$ is not equal to 1 or 9, then $p^2 \mid D$ for some prime $p \neq 3$ (or $p = 3$ and  $p^3 \mid D$) and so any cubic form $$ f(x,y) = ax^3+\frac{3r^2d}{s^2}x^2y+\frac{3rd}{s}xy^2+dy^3$$ with shape $Q$ must have $p^2 \mid d$ because $D \mid 3d$.  By Lemma \ref{maxcrit}, $R(f)$ is not maximal.  
\end{proof}

For discriminants $D= 1$ and $D = 9$, the points $(a,d)^t \in \Z^2$ corresponding to maximal cubic rings have a simple description.  When $D=1$, there is just one integral binary quadratic form of discriminant $D$, up to $\SL_2(\Z)$-equivalence, namely $Q_1(x,y)=xy$.  The cubic rings of shape $Q_1$ have associated binary cubic forms $f(x,y) = ax^3 +dy^3$ for $a,d\in\Z$.

\begin{proposition}\label{1}
If $a,d \in \Z$, then the cubic form $f(x,y) = ax^3 +dy^3$ corresponds to a maximal cubic ring if and only if $a$ and $d$ are both squarefree, $\gcd(a,d) = 1$, and $a^2 \not \equiv d^2$ in $\Z/9\Z$.  
\end{proposition}

\begin{proof}
First suppose $p \neq 3$ is a prime.  Since $\Disc(f) = -27a^2d^2$, if $f$ is non-maximal at~$p$, then $ad \equiv 0$ (mod $p$).  By Lemma \ref{maxcrit}, $f$ is non-maximal at $p$ if either $ a\equiv d \equiv 0$ (mod $p$) or if one of $a$ or $d$ is congruent to 0 (mod $p^2$).  This proves the proposition away from $p = 3$. 

For $p = 3$, note that $f(x,y) \equiv (ax+dy)^3$ $(\mod 3)$.  If $a \equiv d \equiv 0$ (mod 3) or if $a$ or $d$ are divisible by 9, then $R(f)$ is not maximal.  Otherwise, we move the root $ax+dy$ to $x$, and one checks that the coefficient of $y^3$ is 0 (mod 9) precisely when $a \equiv \pm d$ (mod 9), which proves the proposition.
\end{proof}

When $D = 9$, there are two $\SL_2(\Z)$-equivalence classes of primitive integral quadratic forms whose reduced representatives are $Q_{9,r}(x,y) = rx^2+3xy$ for $r = 1,2$.  The cubic rings of shape $Q_{9,r}$ have associated cubic forms $$f(x,y) = ax^3+\frac{r^2d}{3}x^2y+rdxy^2+dy^3.$$ for integers $a$ and $d$ such that $3 \mid d$.  
\begin{proposition}\label{9}
Let $f$ be determined by $a$ and $d$ as above, and set $d' = d/3$ and $a' = r^3d' - 9a$.  Then $\Disc(f) = 9(a'd')^2$ and $f$ is maximal if and only if $\gcd(a',d') = 1$ and both $a'$ and $d'$ are squarefree.
\end{proposition}

\begin{proof}
This can be proved locally at each prime.  For $p = 3$, if $R(f)$ is not maximal at 3, then $\Disc(f) \equiv 0$ (mod 9) which occurs if and only if $d' \equiv a' \equiv 0$ (mod~3).  Conversely, if $d' \equiv 0$ (mod~3), then $R(f)$ is not maximal at 3 by Lemma \ref{maxcrit}.

Next, suppose $p \geq 5$. Again we have $d = 3d'$ for some integer $d'$ and $$f(x,y) = ax^3+r^2d'x^2y+3rd'xy^2+3d'y^3.$$  
If $d' \equiv a \equiv 0$ (mod $p$), then $f$ is imprimitive at $p$; hence $R(f)$ is non-maximal.  If $a \equiv 0  \not \equiv d$ (mod $p$), then $\Disc(f) = -\frac{1}{27}(r^3d^2-27ad)^2 \not \equiv 0$ (mod $p$), so $R(f)$ is maximal and $a' \equiv d' \not \equiv 0$ (mod $p$).  If $d \equiv 0 \not \equiv a$ (mod $p$), then $R(f)$ is maximal precisely when $d \not \equiv 0$ (mod $p^2$).  If both $a$ and $d$ are nonzero (mod $p$), then $\Disc(f) \equiv 0$ (mod $p^2$) if and only if $9a \equiv r^3d'$ (mod~$p$).  In this case, $$f(x,y) \equiv  \frac{d'}{9}\left(rx+3y\right)^3 + kp x^3 \hspace{3mm} \mbox{(mod $p^2$)},$$ where $k$ is a parameter in $\Z$.  By Lemma \ref{maxcrit}, $R(f)$ is non-maximal precisely when $k \equiv 0$ (mod~$p$), i.e., $a' \equiv 0$ (mod $p^2$).  This proves the proposition for $p \geq 5$, and also for the case $p = 2$ and $r = 1$.  If $p = r = 2$, the proof from the previous paragraph does not work, but one checks easily that the proposition still holds.     
\end{proof}

The previous propositions reduce the task of counting cubic fields with shape of square discriminant to estimating the number of lattice points with squarefree and coprime coordinates inside of a rectangular hyperbola.  We can perform this computation with a suitable adaptation of Dirichlet's hyperbola method.  

\vspace{4.25mm}
\begin{proof3}
We have already proved the first part of the theorem, so we consider the second part.  It will simplify expressions if we count the number of cubic fields with bounded \textit{conductor}, instead of bounded discriminant.  Recall, that if the discriminant of a cubic field $K/\Q$ equals $df^2$ with $d$ a fundamental discriminant, then $f$ is called the conductor of $K$.  Thus for $D = \Disc(Q)$ equal to 1 or 9, $M_3(Q,X)$ is the number of cubic fields with shape $Q$ and conductor bounded by $N := (X/3)^{1/2}$.  First we will estimate $M_3(Q,X)$ when $D = 9$.  Since $Q_{9,1}$ and $-Q_{9,2}$ are $\GL_2(\Z)$ equivalent, elements of $L(Q_{9,1})^+$ correspond to cubic forms of shape $Q_{9,1}$ and elements of $L(Q_{9,1})^-$ correspond to cubic rings of shape $Q_{9,2}$.  Thus we may write $Q_9$ for either $Q_{9,1}$ or $Q_{9,2}$ in the following.  By Proposition~\ref{9}, and since $\#\GO_{Q_{9,1}}(\Z) = 4$, 
\begin{align*}
M_3(Q_9,X) &= \frac{1}{2}\cdot\frac{1}{4}\cdot \#\{(a,b)^t \in \Z^2 :  (a,b) = 1, \mu(a)^2 = \mu(b)^2 = 1, a \equiv b \mbox{ (mod } 9), ab \leq N\} \\
&= \frac{1}{2} \#\{(a,b)^t \in \Z^2_{\geq 1} :  (a,b) = 1, \mu(a)^2 = \mu(b)^2 = 1, a \equiv b \mbox{ (mod } 9), ab \leq N\}.
\end{align*}
We define $A_2(N)$ to be the size of this last set.  Following Dirichlet's hyperbola method, we have

\begin{align*}
A_2(N) &= 2\sum_{a \leq \sqrt{N}} \mu(a)^2 \sum_{\substack{b \leq N/a\\  (a,b) = 1 \\ a \equiv b \hspace{1mm}(9)}} \mu^2(b) - \sum_{a \leq \sqrt{N}} \mu(a)^2 \sum_{\substack{b \leq \sqrt{N}\\  (a,b) = 1 \\ a \equiv b \hspace{1mm}(9)}} \mu^2(b)\\
&= 2\sum_{\substack{a \leq \sqrt{N}\\ 3 \nmid a }} \mu(a)^2 \frac{N}{a} \frac{6}{\pi^2} \frac{\phi(a)}{a} \frac{9}{8} \frac{1}{9} \prod_{p | a} \frac{p^2}{p^2 -1} - \sum_{\substack{a \leq \sqrt{N}\\ 3 \nmid a}} \mu(a)^2 \sqrt{N} \frac{6}{\pi^2}\frac{\phi(a)}{a}\frac{1}{9}\frac{9}{8} + O(\sqrt{N}) \\
& = \frac{6N}{4\pi^2} \sum_{\substack{a \leq \sqrt{N}\\ 3\nmid a}}\mu(a)^2 \prod_{p | a}\frac{1}{p+1} - \frac{6\sqrt{N}}{8\pi^2} \sum_{\substack{a \leq \sqrt{N}\\ 3\nmid a}}\mu(a)^2 \prod_{p | a}\frac{p}{p+1} +O(\sqrt{N}).
\end{align*}
We use Perron's formula to estimate both of these sums.  For the first sum, define the function 
$$f(s) = \sum_{\substack{a \geq 1\\ 3 \nmid a}} \mu(a)^2 \prod_{p| a} \frac{1}{p+1} = \prod_{p \neq 3} \left(1 + \frac{1}{p+1}p^{-s}\right)$$ which converges for Re$(s) > 0$.  Also define $h(s) = \frac{f(s)}{\zeta(s+1)}$, which converges for Re$(s) > -1/2$.  By Perron's formula, $$\sum_{\substack{a \leq \sqrt{N}\\ 3\nmid a}}\mu(a)^2 \prod_{p | a}\frac{1}{p+1} = \frac{1}{2\pi i} \int_{c-i\infty}^{c+i\infty} h(s)\zeta(s+1)\sqrt{N}^ss^{-1}ds$$ for any large $c$.  This integral can be estimated by shifting the contour and using Cauchy's formula, which will pick up the residue of $h(s)\zeta(s+1)\sqrt{N}^ss^{-1}$ at $s = 0$.  Using Taylor series, we compute this residue to be $$h(0)\log(\sqrt{N}) + h'(0)+\gamma h(0).$$  The same technique also works for the second sum in our formula for $A_2(N)$, and the residue turns out to be $h(0)\sqrt{N}$.  Altogether, we obtain
$$A_2(N) =   \frac{CN}{10}\left(\log(N)+\frac{2}{5}\log(3)+2\gamma +6\kappa - 1\right)+ o(N),$$
since $({6}/{\pi^2}) h'(0) = {4C}/{5}$ and ${h'(0)}/{h(0)} = 3\kappa + \log(3)/5$.
Replacing $N$ with $(X/3)^{1/2}$, we obtain the desired formula for $M_3(Q_9,X)$. 

To compute $M_3(Q_1,X)$, note that by Proposition \ref{1} we have $$M_3(Q_1,X) = \frac{\#\GO_{Q_9}(\Z)}{\#\GO_{Q_1}(\Z)}\cdot\frac{1}{2}\left(A(N/3) - 2A_2(N/3)\right) = \frac{1}{2}\left(A(N/3) - 2A_2(N/3)\right),$$ where $A(N)$ has the same definition as $A_2(N)$ except without the congruence condition (mod 9).  We can compute $A(N)$ exactly as before, except now the Euler factor at 3 will not be missing.  Combining this with the formula for $A_2(N/3)$ above gives the estimate for $M_3(Q_1,X)$ stated in the theorem.  

By Lemma \ref{purecub} and Proposition \ref{puremax}, we may add the estimates of $M_3(Q,X)$ for $Q =  Q_1, Q_{9,1},$ and $Q_{9,2}$ to obtain a formula for $N(-3,X)$, and this gives Theorem \ref{pure}.
\end{proof3}
%

\begin{remark}
{ \em There is an elementary way to count the density of discriminants of pure cubic fields.  First, we note that two integers $d, d'$ greater than one give rise to the same pure cubic field $K_d:= \Q(d^{1/3})$ if and only if their quotient or product is a cube in $\Q$.  Furthermore, if $d = ab^2$, with $d$ cube-free, and $a$ and $b$ squarefree, then Dedekind \cite{Ded} computed the discriminant of $K_d$ to be $-3k^2$, where 
\begin{equation*}
k = 
\begin{cases} 3ab & \text{if $a^2 \not \equiv b^2$ (mod 9),}
\\
ab &\text{if $a^2 \equiv b^2$ (mod 9).}
\end{cases}
\end{equation*}
Thus, counting pure cubic fields of bounded discriminant is a matter of counting lattice points with squarefree and coprime coordinates under a hyperbola.  So our general method of computing the number of cubic fields having a fixed quadratic resolvent field and bounded discriminant reduces to the classical method in this special case of pure cubic fields.  Furthermore, we see that Dedekind's pure cubic fields of ``Types 1 and 2'' are exactly the pure cubic fields with shape of discriminant 1 and 9, respectively.}  
\end{remark}





\end{document}